\documentclass[11pt,reqno]{amsart}
\usepackage[dvipdfm,a4paper,top=30mm,bottom=25mm,left=25mm,right=25mm]{geometry}
\usepackage{amsmath,amssymb,mathtools}
\usepackage{comment,mathrsfs,bm}

\usepackage[colorlinks=true,linkcolor=blue,citecolor=blue]{hyperref}

\usepackage{graphicx}

\newtheorem{theorem}{Theorem}[section]
\newtheorem{lemma}[theorem]{Lemma}
\newtheorem{proposition}[theorem]{Proposition}

\theoremstyle{definition}
\newtheorem{definition}[theorem]{Definition}

\theoremstyle{remark}
\newtheorem{remark}[theorem]{Remark}

\numberwithin{equation}{section}

\allowdisplaybreaks[4]

\def\rnum#1{\expandafter{\romannumeral #1}} 
\def\Rnum#1{\uppercase\expandafter{\romannumeral #1}}

\newcommand{\N}{{\mathbb{N}}} 
\newcommand{\Z}{{\mathbb{Z}}} 
\newcommand{\R}{{\mathbb{R}}} 
\newcommand{\C}{{\mathbb{C}}} 

\def\~#1{\widetilde #1}

\def\({\left(}
\def\){\right)}
\def\[{\left[}
\def\]{\right]}
\def\<{\left\langle}
\def\>{\right\rangle}

\begin{document}

\title[Nonlinear Schr\"odinger equation with the inverse-square potential]{Scattering and blow-up dichotomy of the energy-critical nonlinear Schr\"odinger equation with the inverse-square potential}


\author{Masaru Hamano}
\address{Faculty of Science and Engineering, Waseda University 3-4-1 Okubo, Shinjuku-ku, Tokyo 169-8555, JAPAN}
\email{m.hamano3@kurenai.waseda.jp}

\author{Masahiro Ikeda}
\address{Center for Advanced Intelligence Project, Riken, Japan / Department of Mathematics, Faculty of Science and Technology, Keio University, 3-14-1 Hiyoshi, Kohoku-ku, Yokohama, 223-8522, Japan}
\email{masahiro.ikeda@riken.jp / masahiro.ikeda@keio.jp}







\begin{abstract}
In this paper, we consider the energy critical nonlinear Schr\"odinger equation with a repulsive inverse square potential.
In particular, we deal with radial initial data, whose energy is equal to the energy of static solution to the corresponding nonlinear Schr\"odinger equation without a potential.
We investigate time behavior of the radial solutions with such initial data.
\end{abstract}

\maketitle

\tableofcontents


\section{Introduction}

Our problem is
\begin{align}\tag{NLS$_\gamma$} \label{NLS}
	i\partial_tu + \Delta_\gamma u
		= - |u|^{\frac{4}{d-2}}u, \quad (t,x) \in \R \times \R^d,
\end{align}
where $d \geq 3$, $\gamma > - (\frac{d-2}{2})^2$, $\Delta_\gamma = \Delta - \frac{\gamma}{|x|^2}$, and $u : \R \times \R^d \longrightarrow \C$ is an unknown function.
In particular, we consider the Cauchy problem of \eqref{NLS} with initial data
\begin{align}\tag{IC} \label{IC}
	u(0,x) = u_0(x).
\end{align}
Let $- \Delta_\gamma^0$ denote the natural action of $- \Delta + \frac{\gamma}{|x|^2}$ on $C_c^\infty(\R^d \setminus \{0\})$, where $C_c^\infty(\R^d \setminus \{0\})$ is a set of infinitely differentiable functions with a compact support on $\R^d \setminus \{0\}$.
Since $- \Delta_\gamma^0$ satisfies
\begin{align}
	\<-\Delta_\gamma^0 f, f\>_{L^2}
		= \int_{\R^d}\[|\nabla f(x)|^2 + \frac{\gamma}{|x|^2}|f(x)|^2\]dx
		= \int_{\R^d}\left|\nabla f(x) + \frac{\sigma}{|x|^2}xf(x)\right|^2dx \label{124}
\end{align}
for $\gamma > - (\frac{d-2}{2})^2$ and any $f \in C_c^\infty(\R^d \setminus \{0\})$, where $\sigma := \frac{d-2}{2} - \frac{1}{2}\sqrt{(d-2)^2 + 4\gamma}$, $- \Delta_\gamma^0$ is a positive semi-definite symmetric operator for $\gamma > - (\frac{d-2}{2})^2$.
We consider $- \Delta_\gamma$ as the self-adjoint extension of $-\Delta_\gamma^0$.
However, it is known that the extension is not unique for $- (\frac{d-2}{2})^2 < \gamma < - (\frac{d-2}{2})^2 + 1$ (see \cite{KalSchWalWus75}).
In this case, we take $- \Delta_\gamma$ as the Friedrichs extension among possible extensions.
Since the operators $- \Delta_\gamma$ and $1 - \Delta_\gamma$ are non-negative, these operators are well-defined on the domain
\begin{align*}
	H_\gamma^1(\R^d)
		:= \{f \in H^1(\R^d) : \|f\|_{H_\gamma^1} < \infty\}
\end{align*}
with a norm $\|f\|_{{H}_\gamma^1}^2 := \|(1-\Delta_\gamma)^\frac{1}{2}f\|_{L^2}^2 = \|f\|_{L^2}^2 + \|(-\Delta_\gamma)^\frac{1}{2}f\|_{L^2}^2$.
We notice that
\begin{align*}
	\|f\|_{\dot{H}_\gamma^1}^2
		:= \|(-\Delta_\gamma)^\frac{1}{2}f\|_{L^2}^2
		= \|\nabla f\|_{L^2}^2 + \int_{\R^d}\frac{\gamma}{|x|^2}|f(x)|^2dx.
\end{align*}

The solution $u$ formally satisfies the following mass and energy :
\begin{align}
	M[u(t)]
		& := \|u(t)\|_{L^2}^2
		= M[u_0], \label{122} \\
	E_\gamma [u(t)]
		& := \frac{1}{2}\|u(t)\|_{\dot{H}_\gamma^1}^2 - \frac{1}{2^\ast}\|u(t)\|_{L^{2^\ast}}^{2^\ast}
		= E_\gamma[u_0], \label{123}
\end{align}
where $2^\ast := \frac{2d}{d-2}$.

Here, we consider \eqref{NLS} with $\gamma = 0$, that is,
\begin{align}\tag{NLS$_0$} \label{NLS0}
	i\partial_t u + \Delta u
		= -|u|^\frac{4}{d-2}u, \qquad (t,x) \in \R \times \R^d.
\end{align}
\eqref{NLS0} is invariant under the scale transformation :
\begin{align*}
	u(t,x)
		\mapsto u_\lambda(t,x) := \lambda^\frac{d-2}{2}u(\lambda^2t,\lambda x)
\end{align*}
for $\lambda > 0$.
In other words, if $u$ is a solution to \eqref{NLS0}, then $u_\lambda$ is also a solution to \eqref{NLS0}.
(To tell the truth, \eqref{NLS} is also invariant under the same scale transformation.)
The scale transformation change an initial data $u_0(x)$ to $\lambda^\frac{d-2}{2}u_0(\lambda x)$.
For the changing, $\|\lambda^\frac{d-2}{2}u_0(\lambda \,\cdot\,)\|_{\dot{H}^1} = \|u_0\|_{\dot{H}^1}$ holds and hence, \eqref{NLS} (with nonlinear power $\frac{d+2}{d-2}$) is called energy critical.

We turn to time behavior of the solution $u$ to \eqref{NLS}.
We define scattering and blow-up to \eqref{NLS}.

\begin{definition}[Scattering and blow-up]
Let $T_{\max}$ denote the maximal forward (resp. backward) existence time of a solution $u$ to \eqref{NLS}.
\begin{itemize}
\item (Scattering)
A solution $u$ to \eqref{NLS} scatters in positive (resp. negative) time if $T_{\max} = \infty$ (resp. $T_{\min} = -\infty$) and there exist $\psi_+ \in H^1(\R^d)$ (resp. $\psi_- \in H^1(\R^d)$) such that
\begin{align*}
	\lim_{t\rightarrow\infty}\|u(t) - e^{it\Delta_\gamma}\psi_+\|_{H_x^1}
		= 0
	\quad
	\(\text{resp. }
	\lim_{t\rightarrow-\infty}\|u(t) - e^{it\Delta_\gamma}\psi_-\|_{H_x^1}
		= 0\).
\end{align*}
If $u$ scatters in positive and negative time, then we say for simplicity that $u$ scatters.
\item (Blow-up)
A solution $u$ to \eqref{NLS} blows up in positive (resp. negative) time if $T_{\max} < \infty$ (resp. $T_{\min} > - \infty$).
If $u$ blows up in positive and negative time, then we say for simplicity that $u$ blows up.
\end{itemize}
\end{definition}

\eqref{NLS} has also the following explicit static solution
\begin{align}
	W_\gamma(x)
		= [d(d-2)\beta^2]^\frac{d-2}{4}\left(\frac{|x|^{\beta-1}}{1+|x|^{2\beta}}\right)^\frac{d-2}{2}, \label{125}
\end{align}
where $\beta = \sqrt{1+\frac{4\gamma}{(d-2)^2}}$.
That is, $W_\gamma$ satisfies
\begin{align}\tag{SP} \label{SP}
	- \Delta_\gamma W_\gamma
		= |W_\gamma|^\frac{4}{d-2}W_\gamma.
\end{align}
For $\gamma \in (-\frac{(d-2)^2}{4},0]$, $W_\gamma$ is characterized by the next Sharp Sobolev embedding.

\begin{proposition}[Sharp Sobolev embedding, \cite{KilMiaVisZhaZhe17}]\label{Sharp Sobolev embedding}
Sharp Sobolev embedding :
\begin{align}\label{001}
	\|f\|_{L^{2^\ast}}
		\leq C_S(\gamma)\|f\|_{\dot{H}_\gamma^1}
\end{align}
holds for any $f \in \dot{H}^1(\R^d)$, where $C_S(\gamma)$ denotes the best constant, that is,
\begin{align*}
	C_S(\gamma)
		:= \sup_{f \in H_\gamma^1(\R^d) \setminus \{0\}}\frac{\|f\|_{L^{2^\ast}}}{\|f\|_{\dot{H}_\gamma^1}}.
\end{align*}
If $\gamma \in(-\frac{(d-2)^2}{4},0]$, then the equality in \eqref{001} is attained by the solution $W_\gamma$.
On the other hand, if $\gamma > 0$, then $C_S(\gamma) = C_S(0)$ and the equality in \eqref{001} is never attained by non-trivial functions.
\end{proposition}

\subsection{Main results}

By using the characterization of $W_\gamma$, the following known results are proved.
First, we list known results for \eqref{NLS0}.
It is known that time behavior of solutions to \eqref{NLS0} with initial data, whose energy is less than the energy of $W_0$.

\begin{theorem}[{Dodson \cite{Dod19}, Kenig--Merle \cite{KenMer06}, Killip--Visan \cite{KilVis10}}]
Let $d \geq 3$.
Assume that $u_0 \in \dot{H}^1(\R^d)$ satisfies $E_0[u_0] < E_0[W_0]$.
Then, the following hold :
\begin{itemize}
\item[(1)]
If $\|u_0\|_{\dot{H}^1} < \|W_0\|_{\dot{H}^1}$ and $u_0 \in \dot{H}_\text{rad}^1(\R^d)$ when $d = 3$, then the solution $u$ to \eqref{NLS0} with \eqref{IC} scatters.
\item[(2)]
If $\|u_0\|_{\dot{H}^1} > \|W_0\|_{\dot{H}^1}$ and ``$xu_0 \in L^2(\R^d)$ or $u_0 \in \dot{H}_\text{rad}^1(\R^d)$'', then the solution $u$ to \eqref{NLS0} with \eqref{IC} blows up.
\end{itemize}
\end{theorem}

It is known that time behavior of solutions to \eqref{NLS0} with initial data, whose energy is equal to that of $W_0$.

\begin{theorem}[Campos--Farah--Roudenko \cite{CamFarRou22}, Duyckaerts--Merle \cite{DuyMer09}, Li--Zhang \cite{LiZha09}]
Let $d \geq 3$.
There exist two radial solutions $W_0^\pm \in \dot{H}^1(\R^d)$ to \eqref{NLS0} satisfying the following properties :
\begin{itemize}
\item[(1)]
$E_0[W_0^+] = E_0[W_0^-] = E_0[W_0]$ and
\begin{align*}
	\lim_{t\rightarrow \infty}\|W_0^\pm(t) - W_0\|_{\dot{H}^1}
		= 0.
\end{align*}
(We note that $W_0^\pm$ is defined on $[0,\infty)$.)
\item[(2)]
$\|W_0^+\|_{\dot{H}^1} < \|W_0\|_{\dot{H}^1}$ and $W_0^+$ scatters in negative time,
\item[(3)]
$\|W_0^-\|_{\dot{H}^1} > \|W_0\|_{\dot{H}^1}$ and if $d \geq 5$, then $W_0^-$ blows up in negative time.
\end{itemize}
\end{theorem}

\begin{theorem}[Campos--Farah--Roudenko \cite{CamFarRou22}, Duyckaerts--Merle \cite{DuyMer09}, Li--Zhang \cite{LiZha09}]
Let $d \geq 3$.
Assume that $u_0 \in \dot{H}_{\text{rad}}^1(\R^d)$ satisfies $E_0[u_0] = E_0[W_0]$.
Then, the following hold :
\begin{itemize}
\item[(1)]
If $\|u_0\|_{\dot{H}^1} < \|W_0\|_{\dot{H}^1}$, then the solution $u$ to \eqref{NLS0} with \eqref{IC} scatters or $u = W_0^+$ up to the symmetries.
\item[(2)]
If $\|u_0\|_{\dot{H}^1} = \|W_0\|_{\dot{H}^1}$, then the solution $u$ to \eqref{NLS0} with \eqref{IC} is $W_0$ up to the symmetries.
\item[(3)]
If $\|u_0\|_{\dot{H}^1} > \|W_0\|_{\dot{H}^1}$ and $u_0 \in L^2(\R^d)$, then the solution $u$ to \eqref{NLS0} with \eqref{IC} blows up or $u = W_0^-$ up to the symmetries.
\end{itemize}
\end{theorem}

We also list the known results for intercritical case (which has nonlinear power in $(1+\frac{4}{d},1+\frac{4}{d-2})$) and mass critical case (which has nonlinear power $1+\frac{4}{d}$).
In the intercritical case, we cite \cite{AkaNaw13, AroDodMur20, DodMur17, DodMur18, DuWuZha16, DuyHolRou10, FanXieCaz11, Gla77, HolRou08, HolRou12, OgaTsu911, OgaTsu912, Zhe18} for below the ground state results and \cite{CamFarRou22, DuyRou10, Guo11, GusInu221, GusInu22} for threshold results.
In the mass critical case, we cite \cite{Dod15}. 

Next, we introduce known results for \eqref{NLS} with $\gamma \neq 0$.
It is known that time behavior of solutions to \eqref{NLS} with initial data, whose energy is less than the energy of $W_\gamma$.

\begin{theorem}[Kai \cite{Kai17, Kai20}, Killip--Miao--Visan--Zhang--Zheng \cite{KilMiaVisZhaZhe17}]\label{Below the Talenti}
Let $d \geq 3$ and $\gamma > -\frac{(d-2)^2}{4} + (\frac{d-2}{d+2})^2$.
Assume that $u_0 \in \dot{H}_{\gamma}^1(\R^d)$ satisfies $E_\gamma[u_0] < E_{\gamma \wedge 0}[W_{\gamma \wedge 0}]$, where $a \wedge b := \min\{a,b\}$.
Then, the following hold :
\begin{itemize}
\item[(1)]
If $3 \leq d \leq 6$, $\|u_0\|_{\dot{H}_\gamma^1} < \|W_{\gamma \wedge 0}\|_{\dot{H}_{\gamma \wedge 0}^1}$, and $u_0$ is radially symmetric for $d = 3$, then the solution $u$ to \eqref{NLS} with \eqref{IC} scatters.
\item[(2)]
If $\|u_0\|_{\dot{H}_\gamma^1} > \|W_{\gamma \wedge 0}\|_{\dot{H}_{\gamma \wedge 0}^1}$ and ``$xu_0 \in L^2(\R^d)$ or $u_0 \in L_\text{rad}^2(\R^d)$'', then the solution $u$ to \eqref{NLS} with \eqref{IC} blows up.
\end{itemize}
\end{theorem}

It is known that time behavior of solutions to \eqref{NLS} with initial data, whose energy is equal to that of $W_\gamma$.

\begin{theorem}[Kai \cite{Kai17}]\label{Special solutions}
Let $3 \leq d \leq 6$ and $-\frac{(d-2)^2}{4} + (\frac{2(d-2)}{d+2})^2 < \gamma < 0$.
There exist two radial solutions $W_\gamma^\pm \in \dot{H}^1(\R^d)$ to \eqref{NLS} satisfying the following properties :
\begin{itemize}
\item[(1)]
$E_\gamma[W_\gamma^+] = E_\gamma[W_\gamma^-] = E_\gamma[W_\gamma]$ and
\begin{align*}
	\lim_{t\rightarrow \infty}\|W_\gamma^\pm(t) - W_\gamma\|_{\dot{H}^1}
		= 0.
\end{align*}
(We note that $W_\gamma^\pm$ is defined on $[0,\infty)$.)
\item[(2)]
$\|W_\gamma^+\|_{\dot{H}_\gamma^1} < \|W_\gamma\|_{\dot{H}_\gamma^1}$ and $W_\gamma^+$ scatters in negative time,
\item[(3)]
$\|W_\gamma^-\|_{\dot{H}_\gamma^1} > \|W_\gamma\|_{\dot{H}_\gamma^1}$ and if $d = 5, 6$, then $W_\gamma^-$ blows up in negative time.
\end{itemize}
\end{theorem}

\begin{theorem}[Kai \cite{Kai17}]\label{Attractive result}
Let $3 \leq d \leq 6$ and $-\frac{(d-2)^2}{4} + (\frac{2(d-2)}{d+2})^2 < \gamma < 0$.
Assume that $u_0 \in \dot{H}_{\text{rad}}^1(\R^d)$ satisfies $E_\gamma[u_0] = E_\gamma[W_\gamma]$.
Then, the following hold :
\begin{itemize}
\item[(1)]
If $\|u_0\|_{\dot{H}_\gamma^1} < \|W_\gamma\|_{\dot{H}_\gamma^1}$, then the solution $u$ to \eqref{NLS} with \eqref{IC} scatters or $u = W_\gamma^+$ up to the symmetries.
\item[(2)]
If $\|u_0\|_{\dot{H}_\gamma^1} = \|W_\gamma\|_{\dot{H}_\gamma^1}$, then the solution $u$ to \eqref{NLS} with \eqref{IC} is $W_\gamma$ up to the symmetries.
\item[(3)]
If $\|u_0\|_{\dot{H}_\gamma^1} > \|W_\gamma\|_{\dot{H}_\gamma^1}$ and $u_0 \in L^2(\R^d)$, then the solution $u$ to \eqref{NLS} with \eqref{IC} blows up in $d = 3, 4$ and ``blows up or $u = W_\gamma^-$ up to the symmetries in $d = 5, 6$''.
\end{itemize}
\end{theorem}

We also list the known results for intercritical case.
We cite \cite{KilMurVisZhe17, LuMiaMur18} for below the ground state results and \cite{MiaMurZhe21} for a threshold result.
In situations different from having inverse square potential, threshold results are given in \cite{ArdHamIke22, ArdInu22, CamMur22, CamPas22, DuyLanRou22, GusInu222, HamKikWat22}.

We are interested in that what happens if we change $E_\gamma[u_0] < E_0[W_0]$ to $E_\gamma[u_0] = E_0[W_0]$ for $\gamma > 0$ in Theorem \ref{Below the Talenti}.
Here, we state our main theorem.

\begin{theorem}\label{Main theorem}
Let $d \geq 3$, $\gamma > 0$, and $u_0 \in \dot{H}_{\text{rad}}^1(\R^d)$.
Then, the following hold :
\begin{itemize}
\item[(1)]
If $3 \leq d \leq 6$ and $u_0 \in BW_+$, then the solution $u$ to \eqref{NLS} with \eqref{IC} scatters, where
\begin{align*}
	BW_+
		:= \{f \in \dot{H}_\text{rad}^1(\R^d) : E_\gamma[f] = E_0[W_0]\text{ and }\|f\|_{\dot{H}_\gamma^1} < \|W_0\|_{\dot{H}^1}\}.
\end{align*}
\item[(2)]
If $u_0 \in BW_-$ and $u_0 \in L^2(\R^d)$, then the solution $u$ to \eqref{NLS} with \eqref{IC} blows up, where
\begin{align*}
	BW_-
		:= \{f \in \dot{H}_\text{rad}^1(\R^d) : E_\gamma[f] = E_0[W_0]\text{ and }\|f\|_{\dot{H}_\gamma^1} > \|W_0\|_{\dot{H}^1}\}.
\end{align*}
\end{itemize}
\end{theorem}

\begin{remark}
It is impossible that $u_0$ satisfies $\|u_0\|_{\dot{H}_\gamma^1} = \|W_0\|_{\dot{H}^1}$ under the assumption $E_\gamma[u_0] = E_0[W_0]$.
\end{remark}

\begin{remark}
Unlike Theorem \ref{Attractive result}, the special solutions $W_\gamma^\pm$ (in Theorem \ref{Special solutions}) does not appear in Theorem \ref{Main theorem}.
Compared with results for equations with a repulsive potential in the intercritical case (e.g. \cite{ArdHamIke22, ArdInu22, MiaMurZhe21}), scale parameters arise in the proof and we need to control them.
\end{remark}

The organization of the rest of this paper is as follows :
In Section \ref{Sec: Preliminary}, we collect some notations and tools.
In Section \ref{Sec: Well-posedness}, we recall the well-posedness results of \eqref{NLS}.
In Section \ref{Sec: Modulation Analysis}, we analyze modulation around $W_0$ of function $f$ with $E_\gamma[f] = E_0[W_0]$.
In Section \ref{Sec: Scattering}, we prove scattering part in Theorem \ref{Main theorem}.
In Section \ref{Sec: Blow-up}, we show blow-up part in Theorem \ref{Main theorem}.

\section{Preliminary}\label{Sec: Preliminary}

\subsection{Notation and definition}

For negative $X$ and $Y$, we write $X\lesssim Y$ to denote $X\leq CY$ for some $C>0$. If $X\lesssim Y\lesssim X$, we write $X\sim Y$.
The dependence of implicit constants on parameters will be indicated by subscripts, e.g. $X\lesssim_uY$ denotes $X\leq CY$ for some $C = C(u)$.
We write $a'\in[1,\infty]$ to denote the H\"older dual exponent to $a\in[1,\infty]$, that is, the solution $\frac{1}{a} + \frac{1}{a'} = 1$. \\

$L^p = L^p(\R^d)$, $1\leq p\leq\infty$ denotes usual Lebesgue space.
For Banach space $X$, we use $L^q(I;X)$ to denote the Banach space of functions $f:I\times \R^d \longrightarrow \C$ whose norm is
\begin{align*}
	\|f\|_{L^q(I;X)}
		:= \left(\int_I\|f(t)\|_X^qdt\right)^\frac{1}{q}
		< \infty,
\end{align*}
with the usual modifications when $q=\infty$.
We extend our notation as follows: If a time interval is not specified, then the $t$-norm is evaluated over $(-\infty,\infty)$.
To indicate a restriction to a time subinterval $I\subset(-\infty,\infty)$, we will write as $L^q(I)$.
$W^{s,p}(\R^d) = (1-\Delta)^{-\frac{s}{2}}L^p(\R^d)$ and $\dot{W}^{s,p}(\R^d) = (-\Delta)^{-\frac{s}{2}}L^p(\R^d)$ are the inhomogeneous Sobolev space and the homogeneous Sobolev space, respectively for $s \in \R$ and $p \in [1,\infty]$.
When $p = 2$, we express $W^{s,2}(\R^d) = H^s(\R^d)$ and $\dot{W}^{s,2}(\R^d) = \dot{H}^s(\R^d)$.
Similarly, we define Sobolev spaces with the potential $W_\gamma^{s,p}(\R^d) = (1-\Delta_\gamma)^{-\frac{s}{2}}L^p(\R^d)$, $\dot{W}_\gamma^{s,p}(\R^d) = (-\Delta_\gamma)^{-\frac{s}{2}}L^p(\R^d)$, $W_\gamma^{s,2}(\R^d) = H_\gamma^s(\R^d)$, and $\dot{W}_\gamma^{s,2}(\R^d) = \dot{H}_\gamma^s(\R^d)$.
To denote a space of radial functions, we use $X_\text{rad} := \{f \in X : f \text{ is a radial function}\}$ for a space $X$.
We set also the following function spaces :
\begin{align*}
	S(I)
		& := L_{t,x}^\frac{2(d+2)}{d-2}(I), \qquad
	X^0(I)
		:= L_t^\frac{2(d+2)}{d-2}(I;L_x^\frac{2d(d+2)}{d^2+4}), \\
	\dot{X}_\gamma^1(I)
		& := L_t^\frac{2(d+2)}{d-2}(I;\dot{W}_\gamma^{1,\frac{2d(d+2)}{d^2+4}}), \qquad
	N^0(I)
		:= L_t^2(I;L_x^\frac{2d}{d+2}) + L_t^1(I;L_x^2).
\end{align*}
We define a scaling notations :
\begin{align*}
	f_{[\lambda]}(x)
		:= \lambda^{-\frac{d-2}{2}}f(\lambda^{-1}x)
	\ \text{ and }\ 
	f_{[\theta,\lambda]}(x)
		:= e^{i\theta}\lambda^{-\frac{d-2}{2}}f(\lambda^{-1}x).
\end{align*}

\begin{remark}
We notice that the following estimates hold :
\begin{itemize}
\item
$\|f\|_{S(I)} \lesssim \|f\|_{\dot{X}^1(I)}$ (from the Sobolev embedding),
\item
$\|e^{it\Delta_\gamma}f\|_{\dot{X}^1(\R)} \lesssim \|f\|_{\dot{H}_\gamma^1}$ (from the Strichartz estimate below (Lemma \ref{Strichartz estimates})).
\end{itemize}
\end{remark}

\subsection{Some tools}

In this subsection, we collect somme tools, which is used throughout this paper.

\begin{lemma}[Hardy inequality]\label{Hardy inequality}
Let $d \geq 3$.
For any $f \in \dot{H}^1(\R^d)$, we have
\begin{align*}
	\(\frac{d-2}{2}\)^2\int_{\R^d}\frac{1}{|x|^2}|f(x)|^2dx
		\leq \|f\|_{\dot{H}^1}^2.
\end{align*}
\end{lemma}

\begin{lemma}[Equivalence of Sobolev norm, \cite{KilMiaVisZhaZhe18, MiaSuZhe22}]\label{Equivalence of Sobolev norm}
Let $d \geq 2$ and $\gamma \geq -\frac{(d-2)^2}{4}$.
\begin{itemize}
\item
Let $-d < s < 2 + (d-2)\beta$ for $\beta$ given in \eqref{125} and $s \neq 0$.
If $1 < p < \infty$ satisfies $\max\{0,\frac{s+\sigma}{d}\} < \frac{1}{p} < \min\{1,\frac{d-\sigma}{d},\frac{d+s}{d}\}$, then
\begin{align*}
	\|(-\Delta)^\frac{s}{2}f\|_{L^p}
		\lesssim \|(-\Delta_\gamma)^\frac{s}{2}f\|_{L^p}
\end{align*}
for any $f \in C_c^\infty(\R^d \setminus \{0\})$, where $\sigma$ is given in \eqref{124}.
\item
Let $- 2 - (d-2)\beta < s < d$ and $s \neq 0$.
If $1 < p < \infty$ satisfies $\max\{0,\frac{s}{d},\frac{\sigma}{d}\} < \frac{1}{p} < \min\{1,\frac{d-\sigma+s}{d}\}$, then
\begin{align*}
	\|(-\Delta_\gamma)^\frac{s}{2}f\|_{L^p}
		\lesssim \|(-\Delta)^\frac{s}{2}f\|_{L^p}
\end{align*}
for any $f \in C_c^\infty(\R^d)$.
\end{itemize}
\end{lemma}

\begin{theorem}[Dispersive estimate, \cite{MiaSuZhe22}]\label{Dispersive estimate}
Let $d \geq 3$, $\gamma \geq - (\frac{d-2}{2})^2$, and $\frac{1}{2} \leq \frac{1}{p} < \min\{1,\frac{d-\sigma}{d}\}$, where $\sigma$ is given in \eqref{124}.
Then, we have
\begin{align*}
	\|e^{it\Delta_\gamma}f\|_{L_x^{p'}}
		\lesssim |t|^{-d(\frac{1}{2}-\frac{1}{p})}\|f\|_{L^p}
\end{align*}
for $t \neq 0$.
\end{theorem}

To state the Strichartz estimates, we define the following $\dot{H}^s$-admissible.

\begin{definition}
We say that $(q,r)$ is $\dot{H}^s$-admissible for $0 \leq s \leq 1$ if $(q,r)$ satisfies
\begin{align*}
	\frac{2}{q} + \frac{d}{r}
		= \frac{d}{2} - s
\end{align*}
and belongs to
\begin{align*}
		\left\{
		\begin{array}{ll}
		\hspace{-0.2cm}
			\{(q,r) ; 2 \leq q \leq \infty, \frac{6}{3-2s} \leq r \leq \frac{6}{1-2s}\} & (0 \leq s < \frac{1}{2}, d = 3) \\
		\hspace{-0.2cm}
			\{(q,r) ; \frac{4}{3-2s} < q \leq \infty, \frac{6}{3-2s} \leq r < \infty\} & (\frac{1}{2} \leq s \leq 1, d = 3) \\
		\hspace{-0.2cm}
			\{(q,r) ; 2 \leq q \leq \infty, \frac{2d}{d-2s} \leq r \leq \frac{2d}{d-2s-2}\} & (d \geq 4).
		\end{array}
		\right.
\end{align*}
We denote
\begin{align*}
	\Lambda_s
		:= \{(q,r) : (q,r)\text{ is }\dot{H}^s\text{-admissible}.\}.
\end{align*}
\end{definition}

\begin{theorem}[Strichartz estimates, \cite{BurPlaStaTah03}]\label{Strichartz estimates}
Let $d \geq 3$ and $\gamma > 0$.
If $(q_1,r_1) \in \Lambda_s$, $(q_2,r_2) \in \Lambda_0$, and $t_0 \in \overline{I}$, then the following estimates hold.
\begin{align*}
	\|e^{it\Delta_\gamma}f\|_{L_t^{q_1}L_x^{r_1}}
		\lesssim \|f\|_{\dot{H}^s}, \quad 
	\left\|\int_{t_0}^te^{i(t-s)\Delta_\gamma}F(\cdot,s)ds\right\|_{L_t^{q_1}(I;L_x^{r_1})}
		\lesssim \|F\|_{L_t^{q_2'}(I;L_x^{r_2'})}.
\end{align*}
\end{theorem}

\begin{lemma}[Radial Sobolev inequality, \cite{Str77}]\label{Radial Sobolev inequality}
Let $p \geq 2$.
Then, we have
\begin{align*}
	\|f\|_{L^p(|x| \geq R)}^p
		\lesssim R^{-\frac{(d-1)(p-2)}{2}}\|f\|_{L^2(|x| \geq R)}^\frac{p+3}{2}\|f\|_{\dot{H}^1(|x| \geq R)}^\frac{p-1}{2}
\end{align*}
for any $f \in H_\text{rad}^1(\R^d)$.
\end{lemma}


\begin{lemma}[\cite{Kai17, KilMiaVisZhaZhe17}]\label{Cor of local smoothing}
Let $\gamma > 0$ and $f \in \dot{H}^1(\R^d)$.
Then, we have
\begin{align*}
	& \|\nabla e^{it\Delta_\gamma}f\|_{L_t^\frac{d+2}{d-2}(B_T(\tau);L_x^\frac{d(d+2)}{d^2-d+2}(B_R(z)))} \\
		& \lesssim T^\frac{(d-2)^2}{2(d+2)^2}R^\frac{d^3+4d-16}{2(d+2)^2}\|e^{it\Delta_\gamma}f\|_{S(\R)}^\frac{d-2}{d+2}\|f\|_{\dot{H}^1}^\frac{4}{d+2}
		 + T^\frac{(d-2)^2}{4(d+2)^2}R^\frac{d^3+d^2-12}{2(d+2)^2}\|e^{it\Delta_\gamma}f\|_{S(\R)}^\frac{d-2}{2(d+2)}\|f\|_{\dot{H}^1}^\frac{d+6}{2(d+2)},
\end{align*}
where $B_r(c)$ denotes a ball having a radius $r$ and a center $c$ and the implicit constant is independent of $f$, $R$, $T$, $\tau$, and $z$.
\end{lemma}

To state localized virial identity, we define the following functions for each $R > 0$ :
A cut-off function $\mathscr{X}_R\in C_0^\infty(\R^d)$ is radially symmetric and satisfies
\begin{equation}
	\mathscr{X}_R(x)
		:= R^2\mathscr{X}\left(\frac{x}{R}\right),\ \text{ where }\ 
	\mathscr{X}(x)
		:= \label{126}
		\left\{
		\begin{array}{cl}
		\hspace{-0.2cm}\displaystyle{ |x|^2 } & \quad (0\leq |x|\leq1), \\
		\hspace{-0.2cm}smooth & \quad (1\leq |x|\leq 3), \\
		\hspace{-0.2cm}0 & \quad (3\leq |x|),
		\end{array}
		\right.
\end{equation}
having $\mathscr{X}''(|x|)\leq 2$ for each $x \in \R^d$.
A cut-off function $\mathscr{Y}_R \in C_0^\infty(\R^d)$ is radially symmetric and satisfies
\begin{equation}
	\mathscr{Y}_R(x)
		:= \mathscr{Y}\left(\frac{x}{R}\right),\ \text{ where }\ 
	\mathscr{Y}(x)
		:= \label{127}
		\left\{
		\begin{array}{cl}
		\hspace{-0.2cm} 1 & \quad (0 \leq |x| \leq 1), \\
		\hspace{-0.2cm} smooth & \quad (1 \leq |x| \leq 2), \\
		\hspace{-0.2cm} 0 & \quad (2 \leq |x|).
		\end{array}
		\right.
\end{equation}

\begin{proposition}[Localized virial identity]\label{Virial identity}
Let $w$ be $\mathscr{X}_R$ or $\mathscr{Y}_R$.
For the solution $u(t)$ to \eqref{NLS}, we define
\begin{align*}
	I_w(t)
		:=\int_{\R^d}w(x)|u(t,x)|^2dx.
\end{align*}
Then, it follows that
\begin{align*}
	I_w'(t)
		= 2\text{Im}\int_{\R^d}\overline{u}(t,x)\nabla u(t,x)\cdot\nabla w(x)dx,
\end{align*}
If $u$ is a radial function, then we have
\begin{align*}
	I_w'(t)
		& = 2\text{Im}\int_{\R^d}\frac{x\cdot\nabla u}{r}\overline{u}w'dx, \\
	I_w''(t)
		& = 4\int_{\R^d}w''(|x|)|\nabla u(t,x)|^2dx - \int_{\R^d}F_2|u(t,x)|^{2^\ast}dx \\
		&\hspace{4.0cm} - \int_{\R^d}F_3|u(t,x)|^2dx + 4\int_{\R^d}w'(|x|)\frac{\gamma}{|x|^3}|u(t,x)|^2dx,
\end{align*}
where
\begin{align*}
	F_2(w,|x|)
		& := \frac{4}{d}\left\{w''(|x|) + \frac{d-1}{|x|}w'(|x|)\right\}, \\
	F_3(w,|x|)
		& := w^{(4)}(|x|) + \frac{2(d-1)}{|x|}w^{(3)}(|x|) + \frac{(d-1)(d-3)}{|x|^2}w''(|x|) + \frac{(d-1)(3-d)}{|x|^3}w'(|x|).
\end{align*}
\end{proposition}

\begin{lemma}[Radial linear profile decomposition, \cite{KilMiaVisZhaZhe17}]\label{Linear profile decomposition}
Let $\{f_n\}$ be a bounded sequence in $\dot{H}_\text{rad}^1(\R^d)$.
Passing to a subsequence, there exists $J^\ast \in \{0, \ldots, \infty\}$, profiles $\{\phi^j\}_{j=0}^{J^\ast} \subset \dot{H}_\text{rad}^1(\R^d)$ satisfying $\phi^0 \equiv 0$, $\phi^j \equiv 0$ for any $j \geq 1$ if $J^\ast = 0$, ``$\phi^j \not\equiv 0$ for any $1 \leq j \leq J^\ast$ and $\phi^j \equiv 0$ for any $j \geq J^\ast + 1$ if $1 \leq J^\ast < \infty$'', and $\phi^j \not\equiv 0$ for any $1 \leq j < \infty$ if $J^\ast = \infty$, time shifts $\{t_n^j\} \subset \R$, scale shifts $\{\lambda_n^j\} \subset (0,\infty)$, and remainders $\{R_n^J\} \subset \dot{H}_\text{rad}^1(\R^d)$ such that the following :
\begin{itemize}
\item
(Decomposition)
\begin{align*}
	f_n
		= \sum_{j=1}^J \phi_n^j + R_n^J
\end{align*}
for each $0 \leq J < \infty$ and each $n \in \N$, where $\phi_n^j := (e^{it_n^j\Delta_\gamma}\phi^j)_{[\lambda_n^j]}$,
\item
(Pythagorean decomposition)
\begin{align}
	& \lim_{n \rightarrow \infty}\Bigl\{\|f_n\|_{\dot{H}_\gamma^1}^2 - \sum_{j=1}^J\|\phi_n^j\|_{\dot{H}_\gamma^1}^2 - \|R_n^J\|_{\dot{H}_\gamma^1}^2\Bigr\}
		= 0, \label{134} \\
	& \lim_{n \rightarrow \infty}\Bigl\{\|f_n\|_{L^{2^\ast}}^{2^\ast} - \sum_{j=1}^J\|\phi_n^j\|_{L^{2^\ast}}^{2^\ast} - \|R_n^J\|_{L^{2^\ast}}^{2^\ast}\Bigr\}
		= 0, \label{135} \\
	& \lim_{n \rightarrow \infty}\Bigl\{E_\gamma[f_n] - \sum_{j=1}^JE_\gamma[\phi_n^j] - E_\gamma[R_n^J]\Bigr\} \label{136}
		= 0,
\end{align}
\item
(Smallness property)
\begin{align}
	\lim_{J \rightarrow \infty}\limsup_{n \rightarrow \infty}\|e^{it\Delta_\gamma}R_n^J\|_{S(\R)}
		= 0, \label{115}
\end{align}
\item
(Orthogonality property)
\begin{align}
	\frac{\lambda_n^j}{\lambda_n^k} + \frac{\lambda_n^k}{\lambda_n^j} + \frac{|t_n^j(\lambda_n^j)^2 - t_n^k(\lambda_n^k)^2|}{\lambda_n^j\lambda_n^k}
		\longrightarrow \infty \label{110}
\end{align}
as $n \rightarrow \infty$ for each $j \neq k$.
Moreover, we can assume that $t_n^j \equiv 0$ for any $n$ or $t_n^j \longrightarrow \pm \infty$ as $n \rightarrow \infty$.
\end{itemize}
\end{lemma}

\section{Well-posedness}\label{Sec: Well-posedness}

In this subsection, we recall global existence and scattering results for \eqref{NLS} with a small data.


\begin{theorem}[Small data global existence, \cite{KilVis13}]\label{Small data global existence}
Let $d \geq 3$, $\gamma > 0$, and $u_0 \in \dot{H}^1(\R^d)$.
Assume that $\|u_0\|_{H_\gamma^1} \leq A$ for given $A > 0$.
Then, there exists $\eta > 0$ such that if
\begin{align*}
	\|e^{i(t-t_0)\Delta_\gamma}u_0\|_{\dot{X}_\gamma^1(I)}
		\leq \eta
\end{align*}
for some time interval $I \ni t_0$, then the solution $u$ to \eqref{NLS} with $u(t_0) = u_0$ exists on $I \times \R^d$ and satisfies
\begin{align*}
	\|u\|_{\dot{X}_\gamma^1(I)}
		\lesssim \eta, \qquad
	\|u\|_{L_t^\infty(I;\dot{H}_\gamma^1)}
		\lesssim A + \eta^\frac{d+2}{d-2}.
\end{align*}
\end{theorem}

\begin{theorem}[Stability, \cite{KilMiaVisZhaZhe17}]\label{Stability}
Let $d \geq 3$ and $\gamma > 0$.
Assume that $I$ is a compact time interval and $\~{u}$ satisfies $\|\~{u}\|_{L_t^\infty(I;\dot{H}_\gamma^1)} \leq E$, $\|\~{u}\|_{S(I)} \leq L$, and
\begin{align*}
	i\partial_t \~{u} + \Delta_\gamma \~{u}
		= - |\~{u}|^\frac{4}{d-2}\~{u} + e\ \text{ on }\ I \times \R^d
\end{align*}
for some positive constants $E$, $L$ and function $e$.
Then, there exists $\varepsilon_0 = \varepsilon_0(E,L) > 0$ such that if $t_0 \in I$ and $u_0 \in \dot{H}^1(\R^d)$ satisfies
\begin{align*}
	\|u_0 - \~{u}(t_0)\|_{\dot{H}_\gamma^1} + \|(-\Delta_\gamma)^\frac{1}{2}e\|_{N^0(I)}
		\leq \varepsilon
\end{align*}
for some $0 < \varepsilon < \varepsilon_1$, then the solution $u$ to \eqref{NLS} on $I \times \R^d$ with $u(t_0) = u_0$ exists and satisfies
\begin{align*}
	\|u - \~{u}\|_{\dot{X}_\gamma^1(I)}
		\leq C(E,L)\varepsilon
	\ \text{ and }\ 
	\|u\|_{\dot{X}_\gamma^1(I)}
		\leq C(E,L).
\end{align*}
\end{theorem}

\begin{theorem}[Local theory, \cite{Caz03, KenMer06}]\label{Local theory}
Let $d \geq 3$ and $\gamma > 0$.
Given $u_0 \in \dot{H}^1(\R^d)$ and $t_0 \in \R$, there exists unique maximal life-span solution $u$ to \eqref{NLS} on $I \times \R^d$ with $u(t_0) = u_0$.
In addition, the following hold :
\begin{itemize}
\item
(Local existence)
$I$ is an open interval including $t_0$.
\item
(Scattering)
$\|u\|_{S(I)} < \infty$ holds if and only if $\|u\|_{\dot{X}_\gamma^1(I)} < \infty$ holds.
Moreover, if $u$ satisfies $\|u\|_{S([0,\infty))} < \infty$, then $u$ scatters in positive time.
A similar result holds for negative time.
\item
(Small data theory)
There exists $\eta_0 > 0$ such that if $\|u_0\|_{\dot{H}_\gamma^1} \leq \eta_0$, then $u$ exists globally in time and satisfies $\|u\|_{\dot{X}_\gamma^1(\R)} \lesssim \|u_0\|_{\dot{H}_\gamma^1}$.
\item
(Large data theory)
For given $u_0 \in \dot{H}^1(\R^d)$, there exists an open interval $I_0 \ni 0$ such that $\|e^{it\Delta_\gamma}u_0\|_{\dot{X}_\gamma^1(I_0)} \leq \eta$, where $\eta$ is given Theorem \ref{Small data global existence}.
\item
(Blow-up criterion)
If $T_{\max} < \infty$, then $\|u\|_{S([0,T_{\max}))} = \infty$ holds.
A similar result holds for negative time.
\item
(Uniqueness)
If $\~{u} \in C_t(I_1;\dot{H}^1(\R^d))$ solves \eqref{NLS} with $\~{u}(t_0) = u$, then $I_1 \subset I$ and $\~{u}(t) = u(t)$ for any $t \in I_1$.
\item
(Continuity dependence on initial data)
If $\~{u}$ solves \eqref{NLS} on $I$ with initial data $\~{u}_0$ and satisfies $\|\~{u}\|_{L_t^\infty(I;\dot{H}_\gamma^1)} + \|\~{u}\|_{S(I)} \leq A$ for some $A > 0$, then there exists $\varepsilon_0 = \varepsilon_0(A)$, $C = C(A) > 0$ such that for any $\|\~{u}_0 - u_0\|_{\dot{H}_\gamma^1} = \varepsilon < \varepsilon_0$, the solution $u$ to \eqref{NLS} with initial data $u_0$ is defined on $I$ and satisfies $\|u\|_{S(I)} \leq C$ and $\|\~{u} - u\|_{L_t^\infty(I;\dot{H}_\gamma^1)} \leq C\varepsilon$.
\end{itemize}
\end{theorem}

\begin{theorem}[Conservation laws]
Let $d \geq 3$, $\gamma > 0$, and $u_0 \in \dot{H}^1(\R^d)$.
Then, the solution $u$ to \eqref{NLS} with \eqref{IC} conserves its energy \eqref{123}.
Moreover, if $u_0 \in H^1(\R^d)$, then $u$ also conserves its mass \eqref{122}.
\end{theorem}

\begin{proposition}[Final state problem]\label{Existence of a wave operator}
For any $u_+ \in \dot{H}_\text{rad}^1(\R^d)$, there exists $T > 0$ and a solution $u$ to \eqref{NLS} defined on $C_t([T,\infty);\dot{H}^1(\R^d))$ such that
\begin{align*}
	\lim_{t \rightarrow \infty}\|u(t) - e^{it\Delta_\gamma}u_+\|_{\dot{H}^1}
		= 0.
\end{align*}
\end{proposition}

\section{Modulation Analysis}\label{Sec: Modulation Analysis}

We define a function
\begin{align*}
	(\delta(u(t)) = )
		\delta(t)
		:= \|W_0\|_{\dot{H}^1}^2 - \|u(t)\|_{\dot{H}_\gamma^1}^2
\end{align*}
and a set
\begin{align*}
	I_0
		:= \{t \in (T_{\min},T_{\max}) : |\delta(t)| < \delta_0\}
\end{align*}
for a given small parameter $\delta_0 > 0$.

\begin{lemma}\label{Variational characterization}
Let $d \geq 3$ and $\gamma > 0$.
Assume that $f \in \dot{H}_\text{rad}^1(\R^d)$ satisfy $E_\gamma[f] = E_0[W_0]$.
For any $\varepsilon > 0$, there exists $\delta_0 = \delta_0(\varepsilon) > 0$ small enough such that if $|\delta(f)| < \delta_0$, then there exists $(\theta,\mu) \in \R / 2\pi\Z \times (0,\infty)$ so that
\begin{align*}
	\|f_{[\theta,\mu]} - W_0\|_{\dot{H}^1}
		< \varepsilon.
\end{align*}
\end{lemma}

\begin{proof}
We prove the claim by contradiction.
Then, there exists $\varepsilon_0 > 0$ such that for any $n \in \N$, there exists $f_n \in \dot{H}_\text{rad}^1(\R^d)$ satisfying $E_\gamma[f_n] = E_0[W_0]$ and $|\delta(f_n)| < \frac{1}{n}$ so that
\begin{align*}
	\inf_{(\theta,\mu) \in \R / 2\pi\Z \times (0,\infty)}\|(f_n)_{[\theta,\mu]} - W_0\|_{\dot{H}^1}
		\geq \varepsilon_0.
\end{align*}
$E_\gamma[f_n] = E_0[W_0]$ and $|\delta(f_n)| < \frac{1}{n}$ deduce $\|f_n\|_{\dot{H}_\gamma^1} \longrightarrow \|W_0\|_{\dot{H}^1}$ and $\|f_n\|_{L^{2^\ast}} \longrightarrow \|W_0\|_{L^{2^\ast}}$ as $n \rightarrow \infty$.
Thus, applying Lemma \ref{Linear profile decomposition} and Proposition \ref{Sharp Sobolev embedding}, we have
\begin{align*}
	C_S(0)
		& = \lim_{n\rightarrow \infty}\frac{\|f_n\|_{L^{2^\ast}}}{\|f_n\|_{\dot{H}_\gamma^1}} \\
		& = \lim_{J\rightarrow J^\ast}\lim_{n\rightarrow \infty}\frac{(\sum_{j=1}^J\|\phi_n^j\|_{L^{2^\ast}}^{2^\ast} + \|R_n^J\|_{L^{2^\ast}}^{2^\ast})^\frac{1}{2^\ast}}{(\sum_{j=1}^J\|\phi^j\|_{\dot{H}_\gamma^1}^2 + \|R_n^J\|_{\dot{H}_\gamma^1}^2)^\frac{1}{2}} \\
		& \leq C_S(0)\lim_{J\rightarrow J^\ast}\lim_{n\rightarrow \infty}\frac{(\sum_{j=1}^J\|\phi^j\|_{\dot{H}_\gamma^1}^{2^\ast} + \|R_n^J\|_{\dot{H}_\gamma^1}^{2^\ast})^\frac{1}{2^\ast}}{(\sum_{j=1}^J\|\phi^j\|_{\dot{H}_\gamma^1}^2 + \|R_n^J\|_{\dot{H}_\gamma^1}^2)^\frac{1}{2}},
\end{align*}
which implies that
\begin{align*}
	\Bigl(\sum_{j=1}^{J^\ast}\|\phi^j\|_{\dot{H}_\gamma^1}^{2^\ast} + \lim_{J\rightarrow J^\ast}\lim_{n\rightarrow \infty}\|R_n^J\|_{\dot{H}_\gamma^1}^{2^\ast}\Bigr)^\frac{1}{2^\ast}
		\geq \Bigl(\sum_{j=1}^{J^\ast}\|\phi^j\|_{\dot{H}_\gamma^1}^2 + \lim_{J\rightarrow J^\ast}\lim_{n\rightarrow \infty}\|R_n^J\|_{\dot{H}_\gamma^1}^2\Bigr)^\frac{1}{2}
\end{align*}
and hence, $J^\ast = 1$ and $\lim_{J\rightarrow J^\ast}\lim_{n\rightarrow \infty}\|R_n^J\|_{\dot{H}_\gamma^1} = \lim_{n\rightarrow \infty}\|f_n - \phi_n^1\|_{\dot{H}_\gamma^1} = 0$.
That is, $f_n$ can be written as
\begin{align*}
	f_n
		= (e^{it_n\Delta_\gamma}\phi)_{[\mu_n]} + R_n
	\ \text{ and }\ 
	\lim_{n \rightarrow \infty}\|R_n\|_{\dot{H}_\gamma^1}
		= 0.
\end{align*}
Here, we claim $t_n \equiv 0$ for any $n \in \N$.
We assume for contradiction that $|t_n| \longrightarrow \infty$ as $n \rightarrow \infty$.
Then, we have
\begin{align*}
	\lim_{n\rightarrow \infty}\|f_n\|_{L^{2^\ast}}^{2^\ast}
		= \lim_{n\rightarrow \infty}\|e^{it_n\Delta_\gamma}\phi\|_{L^{2^\ast}}^{2^\ast} + \lim_{n\rightarrow \infty}\|R_n\|_{L^{2^\ast}}^{2^\ast}
		= 0
\end{align*}
from Proposition \ref{Sharp Sobolev embedding} and Theorem \ref{Dispersive estimate}.
However, this contradicts $E_\gamma[f_n] = E_0[W_0]$.
Therefore, $t_n \equiv 0$ for each $n \in \N$ and hence, $\phi$ is a optimizer of $C_S(\gamma)$, which is a contradiction.
\end{proof}

Next, we introduce a quadratic form $Q$ based on the energy decomposition near $W_0$ :
\begin{align}
	E_\gamma[W_0 + g]
		= E_0[W_0] + Q(g) + \frac{1}{2}\int_{\R^d}\frac{\gamma}{|x|^2}|W_0 + g|^2dx + \mathcal{O}(\|g\|_{\dot{H}^1}^3), \label{106}
\end{align}
where
\begin{align*}
	Q(g)
		:= \frac{1}{2}\|g\|_{\dot{H}^1}^2 - \frac{1}{2}\int_{\R^d}W_0^\frac{4}{d-2}\left[\frac{d+2}{d-2}(\text{Re}\,g)^2 + (\text{Im}\,g)^2\right]dx.
\end{align*}
The operator $Q$ can be also written as $Q(g) = \frac{1}{2}\<\mathcal{L}g,ig\>_{L^2}$ by using the linearized operator $\mathcal{L}$ around $W_0$ :
\begin{align*}
	\mathcal{L}
		:=
		\begin{pmatrix}
			0 & \Delta + W_0^{2^\ast-2}\, \\
			- \Delta - (2^\ast-1)W_0^{2^\ast-2} & 0
		\end{pmatrix},
\end{align*}
where we identify $\begin{pmatrix} a \\ b \end{pmatrix} \in \R^2$ with $a + ib \in \C$.
Then, if $v := u - W_0$ for a solution $u$ to \eqref{NLS}, we have
\begin{align*}
	\partial_t v + \mathcal{L}(v) + R(v)
		= 0,
\end{align*}
where $R(v)$ is defined as
\begin{align*}
	R(v)
		:= i(2^\ast-1)W_0^{2^\ast-2}v_1 - W_0^{2^\ast-2}v_2 + iW_0^{2^\ast-1} + i\frac{\gamma}{|x|^2}(v + W_0) - i|v + W_0|^{2^\ast-2}(v + W_0).
\end{align*}
We define $W_1 := - \left.\frac{d}{d\mu}\right|_{\mu=1}(W_0)_{[\mu]} = \frac{d-2}{2}W_0 + x\cdot\nabla W_0$.
By the direct calculation, we see that $W_0$, $iW_0$, $W_1$ form three orthogonal directions in the real Hilbert space $\dot{H}^1(\R^d)$.
We get $Q(W_0) < 0$ and $Q(iW_0) = 0$ immediately.
The equation \eqref{SP} deduces $(\Delta + (2^\ast-1)W_0^{2^\ast-2})W_1 = 0$ and hence, we have $Q(W_1) = -\frac{1}{2}\<(\Delta + (2^\ast-1)W_0^{2^\ast-2})W_1,W_1\>_{L^2} = 0$.
Setting $H := \text{Span}\{W_0, iW_0, W_1\} \subset \dot{H}^1(\R^d)$ and $H^\perp := \{v \in \dot{H}^1(\R^d) ; \<v,w\>_{\dot{H}^1} = 0\text{ for any }w \in H\}$.
Then, $Q(v)$ is positive definite for $v \in H^\perp$, that is, the following lemma holds.

\begin{lemma}[Coercivity of $Q$, \cite{DuyMer09, KaiZenZha22, Rey90}]\label{Coercivity}
For any radial function $v \in H^\perp$, we have
\begin{align*}
	Q(v)
		\sim \|v\|_{\dot{H}^1}^2.
\end{align*}
\end{lemma}

\begin{lemma}[\cite{DuyMer09}]\label{Modulation lemma}
There exists $\delta_0 > 0$ such that for any $f \in \dot{H}_\text{rad}^1(\R^d)$ with $E_\gamma[f] = E_0[W_0]$ and $|\delta(f)| < \delta_0$, there exists $(\theta,\mu) \in \R/ 2\pi\Z \times (0,\infty)$ satisfying $f_{[\theta,\mu]} \perp iW_0, W_1$.
The parameter $(\theta,\mu)$ is unique in $\R / 2\pi \Z \times \R$ and a mapping $f \mapsto (\theta,\mu)$ is $C^1$.
\end{lemma}

The next proposition follows from Lemma \ref{Modulation lemma}.

\begin{proposition}[Modulation]\label{Modulation}
For each sufficiently small $\delta_0 > 0$, there exists $(\theta,\mu) : I_0 \longrightarrow \R/2\pi \Z \times \R$ such that
\begin{align}
	u_{[\theta,\mu]}(t,x)
		= (1+\alpha(t))W_0 + v(t), \label{145}
\end{align}
where
\begin{align*}
	1 + \alpha(t)
		= \frac{1}{\|W_0\|_{\dot{H}^1}^2}\<u_{[\theta,\mu]},W_0\>_{\dot{H}^1}
	\ \text{ and }\ 
	v(t)
		\in H^\perp.
\end{align*}
\end{proposition}

The parameters given in Proposition \ref{Modulation} have the following relations (Lemmas \ref{Parameters for modulation} and \ref{Parameters for modulation2}).

\begin{lemma}\label{Parameters for modulation}
Let $\gamma > 0$ and $\delta_0 > 0$ be small enough.
Then, we have
\begin{align*}
	\[\int_{\R^d}\frac{\gamma}{|x|^2}|u(t,x)|^2dx\]^\frac{1}{2} + \|v(t)\|_{\dot{H}^1}
		\sim |\alpha(t)|
		\sim \|w(t)\|_{\dot{H}^1}
		\sim |\delta(t)|
\end{align*}
for any $t \in I_0$, where $w(t) := \alpha(t)W_0 + v(t)$ and $(\theta,\lambda)$ is given in Proposition \ref{Modulation}.
\end{lemma}

\begin{proof}
By the taking $w$, we get
\begin{align*}
	\alpha(t)
		= \frac{1}{\|W_0\|_{\dot{H}^1}^2}\<w,W_0\>_{\dot{H}^1}
		\in \R
\end{align*}
and hence,
\begin{align}
	|\alpha(t)|
		\lesssim \|w\|_{\dot{H}^1}
		\ll 1 \label{130}
\end{align}
holds.
By the construction, we have
\begin{align}
	\|w\|_{\dot{H}^1}^2
		= |\alpha|^2\|W_0\|_{\dot{H}^1}^2 + \|v\|_{\dot{H}^1}^2 \label{131}
\end{align}
and hence,
\begin{align}
	Q(w)
		& = |\alpha|^2Q(W_0) + Q(v) - \frac{\alpha(d+2)}{d-2}\<W_0,v\>_{\dot{H}^1} \notag \\
		& = |\alpha|^2Q(W_0) + Q(v). \label{128}
\end{align}
Applying \eqref{106} and $E_\gamma[W_0 + w] = E_\gamma[u_{[\theta,\mu]}] = E_0[W_0]$, we have
\begin{align}
	0
		= Q(w) + \frac{1}{2}\int_{\R^d}\frac{\gamma}{|x|^2}|u|^2dx + \mathcal{O}(\|w\|_{\dot{H}^1}^3). \label{129}
\end{align}
It follows from Lemma \ref{Coercivity}, \eqref{130}, \eqref{131}, \eqref{128}, and \eqref{129} that
\begin{align*}
	\|v\|_{\dot{H}^1}^2 + \int_{\R^d}\frac{\gamma}{|x|^2}|u|^2dx
		& \lesssim Q(v) + \int_{\R^d}\frac{\gamma}{|x|^2}|u|^2dx
		\lesssim |\alpha|^2 + \|w\|_{\dot{H}^1}^3 \\
		& \lesssim |\alpha|^2 + |\alpha|^3 + \|v\|_{\dot{H}^1}^3
		\lesssim |\alpha|^2 + \|v\|_{\dot{H}^1}^3.
\end{align*}
Noting $\|v\|_{\dot{H}^1} \ll 1$ holds from \eqref{130} and \eqref{131}, we obtain
\begin{align}
	\|v\|_{\dot{H}^1}^2 + \int_{\R^d}\frac{\gamma}{|x|^2}|u|^2dx
		\lesssim |\alpha|^2. \label{132}
\end{align}
Recalling $Q(W_0) < 0$, we have
\begin{align*}
	|\alpha|^2
		& \sim - Q(W_0)|\alpha|^2
		= Q(v) - Q(w)
		\lesssim \|v\|_{\dot{H}^1}^2 + \int_{\R^d}\frac{\gamma}{|x|^2}|u|^2dx + \mathcal{O}(\|w\|_{\dot{H}^1}^3) \\
		& \lesssim \|v\|_{\dot{H}^1}^2 + \int_{\R^d}\frac{\gamma}{|x|^2}|u|^2dx + \|v\|_{\dot{H}^1}^3 + |\alpha|^3
		\lesssim \|v\|_{\dot{H}^1}^2 + \int_{\R^d}\frac{\gamma}{|x|^2}|u|^2dx + |\alpha|^3.
\end{align*}
Using \eqref{130}, we get
\begin{align}
	|\alpha|^2
		\lesssim \|v\|_{\dot{H}^1}^2 + \int_{\R^d}\frac{\gamma}{|x|^2}|u|^2dx. \label{133}
\end{align}
Collecting \eqref{132}, \eqref{130}, \eqref{131}, and \eqref{133}, we obtain
\begin{align*}
	\|v\|_{\dot{H}^1}^2 + \int_{\R^d}\frac{\gamma}{|x|^2}|u|^2dx
		\lesssim |\alpha|^2
		\lesssim \|w\|_{\dot{H}^1}^2
		\lesssim |\alpha|^2 + \|v\|_{\dot{H}^1}^2
		\lesssim \|v\|_{\dot{H}^1}^2 + \int_{\R^d}\frac{\gamma}{|x|^2}|u|^2dx.
\end{align*}
Finally, we check $|\delta(t)| \sim |\alpha|$.
By the definition $\delta(t)$ and $v \in H^\perp$, we have
\begin{align*}
	\delta(t)
		& = \|W_0\|_{\dot{H}^1}^2 - \|u\|_{\dot{H}_\gamma^1}^2
		= \|W_0\|_{\dot{H}^1}^2 - \|u_{[\theta,\mu]}\|_{\dot{H}^1}^2 - \int_{\R^d}\frac{\gamma}{|x|^2}|u|^2dx \\
		& = \|W_0\|_{\dot{H}^1}^2 - \|(1 + \alpha)W_0 + v\|_{\dot{H}^1}^2 - \int_{\R^d}\frac{\gamma}{|x|^2}|u|^2dx \\
		& = - (2\alpha + \alpha^2)\|W_0\|_{\dot{H}^1}^2 - \int_{\R^d}\frac{\gamma}{|x|^2}|u|^2dx - \|v\|_{\dot{H}^1}^2.
\end{align*}
Therefore, we obtain
\begin{align*}
	|\delta(t)|
		\lesssim |\alpha| + |\alpha|^2
		\lesssim |\alpha|.
\end{align*}
Thus, we have
\begin{align*}
	|\alpha|
		\sim 2\|W_0\|_{\dot{H}^1}^2|\alpha|
		= \left|\delta(t) +  |\alpha|^2\|W_0\|_{\dot{H}^1}^2 + \int_{\R^d}\frac{\gamma}{|x|^2}|u|^2dx + \|v\|_{\dot{H}^1}^2\right|
		\lesssim |\delta(t)| + |\alpha|^2
		\lesssim |\delta(t)|
\end{align*}
and hence, $|\alpha| \lesssim |\delta|$ holds.
\end{proof}

\begin{lemma}\label{Parameters for modulation2}
Let $\delta_0 > 0$ be small enough.
Then, we have
\begin{align*}
	|\alpha'(t)|
		+ |\theta'(t)|
		+ \left|\frac{\mu'(t)}{\mu(t)}\right|
		\lesssim \mu(t)^2|\delta(t)|
\end{align*}
for any $t \in I_0$, where $(\theta,\mu)$ and $\alpha$ are given in Proposition \ref{Modulation}.
\end{lemma}

\begin{proof}
This lemma follows from the argument in \cite{DuyMer09} combined with Lemma \ref{Parameters for modulation}.
\end{proof}

\section{Scattering}\label{Sec: Scattering}

In this section, we prove the scattering result in Theorem \ref{Main theorem}.

\subsection{Gradient separation}

We see that $BW_+$ and $BW_-$ are invariant for the flow of \eqref{NLS}.

\begin{lemma}\label{Invariant sets}
Let $\gamma > 0$ and $u_0 \in \dot{H}_\text{rad}^1(\R^d)$.
\begin{itemize}
\item[(1)]
If $u_0 \in BW_+$, then the solution $u$ to \eqref{NLS} with \eqref{IC} satisfies $u(t) \in BW_+$ for any $t \in (T_{\min},T_{\max})$.
\item[(2)]
If $u_0 \in BW_-$, then the solution $u$ to \eqref{NLS} with \eqref{IC} satisfies $u(t) \in BW_-$ for any $t \in (T_{\min},T_{\max})$.
\end{itemize}
\end{lemma}

\begin{proof}
For contradiction, we assume that there exists $t_0 \in (T_{\min},T_{\max})$ such that $\|u(t_0)\|_{\dot{H}_\gamma^1} = \|W_0\|_{\dot{H}^1}$.
Combined with $E_\gamma[u_0] = E_0[W_0]$, we get $\|u(t_0)\|_{L^{2^\ast}} = \|W_0\|_{L^{2^\ast}}$, which implies that $u(t_0)$ is a optimizer of \eqref{001}.
This is a contradiction.
The same argument yields the claim (2).
\end{proof}

\begin{lemma}[Coercivity of energy, \cite{Kai17}]\label{Coercivity of energy}
Let $f \in \dot{H}_\gamma^1(\R^d)$.
If $f$ satisfies $\|f\|_{\dot{H}_\gamma^1} \leq \|W_0\|_{\dot{H}^1}$, then
\begin{align*}
	\frac{1}{d}\|f\|_{\dot{H}_\gamma^1}^2
		\leq E_\gamma[f]
		\leq \frac{1}{2}\|f\|_{\dot{H}_\gamma^1}^2.
\end{align*}
\end{lemma}

\subsection{Existence of a soliton-like solution}

In this subsection, we prove that if the scattering part in Theorem \ref{Main theorem} does not hold, there exists a soliton-like solution $u$.

\begin{proposition}\label{Precompact of critical solution}
Let $\gamma > 0$ and $u_0 \in BW_+$.
Then, the solution $u$ to \eqref{NLS} exists globally in time.
Moreover, if $\|u\|_{S([0,\infty))} = \infty$, then there exists $\lambda(t) : [0,\infty) \longrightarrow (0,\infty)$ such that $\{u(t)_{[\lambda(t)]} : t \in [0,\infty)\}$ is precompact in $\dot{H}^1(\R^d)$.
A similar result holds on $(-\infty,0]$.
\end{proposition}

We split Proposition \ref{Precompact of critical solution} into Lemmas \ref{Precompact of critical solution along a sequence}, \ref{Precompact of critical solution without global existence}, and \ref{Global existence}.

\begin{lemma}\label{Precompact of critical solution along a sequence}
Let $\gamma > 0$ and $u_0 \in BW_+$.
Suppose that the solution $u$ to \eqref{NLS} with \eqref{IC} satisfies $\|u\|_{S([0,T_{\max}))} = \infty$.
Then, for each $\{\tau_n\} \subset [0,T_{\max})$, there exists $\{\lambda_n\} \subset (0,\infty)$ such that $\{u(\tau_n)_{[\lambda_n]}\} \subset \dot{H}^1(\R^d)$ is precompact.
\end{lemma}

\begin{proof}
Take any sequence $\{\tau_n\} \subset [0,T_{\max})$.
Applying Lemma \ref{Linear profile decomposition} to $\{u_n(0) := u(\tau_n)\}$, after passing to a subsequence, there exist $J^\ast \in \{0,1,\ldots,\infty\}$, profiles $\{\phi^j\}_{j=0}^{J^\ast} \subset \dot{H}^1(\R^d)$, $\{(\lambda_n^j,t_n^j)\}_{j=1}^{J^\ast} \subset (0,\infty) \times \R$, and $\{R_n^j\}_{j=1}^{J^\ast} \subset \dot{H}^1(\R^d)$ such that
\begin{align*}
	u_n(0)
		= \sum_{j=1}^J\phi_n^j + R_n^J
\end{align*}
for each $0 \leq J \leq J^\ast$ and $J \in \N$, where $\phi_n^j := (e^{it_n^j\Delta_\gamma}\phi^j)_{[\lambda_n^j]}$.

\textbf{(Step 1).}
We exclude $J^\ast = 0$.

If $J^\ast = 0$, then $u(\tau_n) = R_n^0$.
Since $\|e^{i(t-\tau_n)\Delta_\gamma}u(\tau_n)\|_{S(\R)} = \|e^{it\Delta_\gamma}R_n^0\|_{S(\R)} \longrightarrow 0$ as $n \rightarrow \infty$, $\|u\|_{S(\R)} < \infty$ for sufficiently large $n$ from Lemma \ref{Small data global existence}.

\textbf{(Step 2).}
We exclude $J^\ast \geq 2$.

It follows from \eqref{134} and Lemma \ref{Invariant sets} that
\begin{gather}
	\|\phi_n^j\|_{\dot{H}_\gamma^1}^2
		= \|\phi^j\|_{\dot{H}_\gamma^1}^2
		< \sum_{j = 1}^J\|\phi^j\|_{\dot{H}_\gamma^1}^2
		\leq \limsup_{n \rightarrow \infty}\|u(\tau_n)\|_{\dot{H}_\gamma^1}^2
		\leq \|W_0\|_{\dot{H}^1}^2, \label{137} \\
	\limsup_{n \rightarrow \infty}\|R_n^J\|_{\dot{H}_\gamma^1}^2
		< \limsup_{n \rightarrow \infty}\|u(\tau_n)\|_{\dot{H}_\gamma^1}^2
		\leq \|W_0\|_{\dot{H}^1}^2 \label{138}
\end{gather}
for each $1 \leq j \leq J$ and hence, Lemma \ref{Coercivity of energy} gives us
\begin{align*}
	0
		< \frac{1}{d}\|\phi_n^j\|_{\dot{H}_\gamma^1}^2
		\leq E_\gamma[\phi_n^j]
	\ \text{ and }\ 
	0
		\leq \frac{1}{d}\limsup_{n\rightarrow \infty}\|R_n^J\|_{\dot{H}_\gamma^1}^2
		\leq \limsup_{n\rightarrow \infty}E_\gamma[R_n^J],
\end{align*}
which combined with \eqref{136} implies that
\begin{align*}
	\sup_{j}\limsup_{n\rightarrow \infty}E_\gamma[\phi_n^j]
		< E_\gamma[u(\tau_n)]
		= E_0[W_0].
\end{align*}
Take a small positive constant $\delta$ satisfying
\begin{align}
	\limsup_{n\rightarrow \infty}E_\gamma[\phi_n^j]
		\leq E_0[W_0] - \delta \label{108}
\end{align}
for any $1 \leq j \leq J$.
We set a solution $u^j : I^j \times \R^d \longrightarrow \C$ to \eqref{NLS}, which has a initial data $\phi^j$ when $t_n^j \equiv 0$ and satisfies $\|u^j(t) - e^{it\Delta_\gamma}\phi^j\|_{\dot{H}^1} \longrightarrow 0$ as $t \rightarrow \pm \infty$ when $t_n^j \rightarrow \pm \infty$, where $I^j$ denotes the maximal life-span of $u^j$.
Existence of $u^j$ is assured by Theorem \ref{Local theory} and Lemma \ref{Existence of a wave operator}.
Moreover, we set $u_n^j(t,x) := u^j((\lambda_n^j)^{-2}t + t_n^j,x)_{[\lambda_n^j]}$ defined on $I_n^j \times \R^d$.
We note $I_n^j := \{t \in \R : (\lambda_n^j)^{-2}t + t_n^j \in I^j\}$.
By the construction, we have
\begin{align}
	\lim_{n \rightarrow \infty}\|u_n^j(0) - \phi_n^j\|_{\dot{H}^1}
		= 0 \label{107}
\end{align}
in both cases of $t_n^j \equiv 0$ and $t_n^j \longrightarrow \pm\infty$ as $n \rightarrow \infty$.
It follows from \eqref{137}, \eqref{108}, and \eqref{107} that
\begin{align*}
	E_\gamma[u_n^j(0)]
		< E_0[W_0]
	\ \text{ and }\ 
	\|u_n^j(0)\|_{\dot{H}_\gamma^1}
		< \|W_0\|_{\dot{H}^1}
\end{align*}
for sufficiently large $n \in \N$.
Theorem \ref{Below the Talenti} deduces that $u_n^j$ scatters for sufficiently large $n \in \N$.
Applying Theorem \ref{Strichartz estimates}, we have $\|u_n^j\|_{\dot{X}^1(\R)} < \infty$.
Here, we take a function $\psi_\varepsilon^j \in C_c^\infty(\R \times \R^d)$ satisfying
\begin{align*}
	\|u^j - \psi_\varepsilon^j\|_{\dot{X}^1(\R)}
		< \varepsilon
\end{align*}
by the density of $C_c^\infty(\R \times \R^d) \subset \dot{X}^1(\R)$.
From the change of variable, we get
\begin{align}
	\|u_n^j - T_n^j\psi_\varepsilon^j\|_{\dot{X}^1(\R)}
		< \varepsilon, \label{109}
\end{align}
where $T_n^j\psi_\varepsilon^j(t,x) := \psi_\varepsilon^j((\lambda_n^j)^{-2}t + t_n^j,x)_{[\lambda_n^j]}$.
From now on, we discuss two properties of $u_n^j$ :

\textbf{(Claim 1). Boundedness of $u_n^j$ :}
There exists $J_0 \in N$ such that
\begin{align}\label{139}
	\|u_n^j\|_{\dot{X}^1(\R)}
		\lesssim \|\phi^j\|_{\dot{H}_\gamma^1}
		\ \text{ for any }\ J \geq J_0.
\end{align}
It follows from \eqref{137} that $\lim_{j \rightarrow \infty}\|\phi^j\|_{\dot{H}_\gamma^1} = 0$.
Together with \eqref{107}, there exists $J_0 \in \N$ such that $\|u_n^j\|_{\dot{X}^1(\R)} \lesssim \|u_n^j(0)\|_{\dot{H}_\gamma^1} \lesssim \|\phi^j\|_{\dot{H}_\gamma^1}$ for any $j \geq J_0$ from Theorem \ref{Small data global existence}.

\textbf{(Claim 2). Orthogonality of $u_n^j$ and $u_n^k$ $(j \neq k)$} :
\begin{align}
	\|\nabla u_n^j \nabla u_n^k\|_{L_t^\frac{d+2}{d-2}L_x^\frac{d(d+2)}{d^2+4}} + \|u_n^j u_n^k\|_{L_{t,x}^\frac{d+2}{d-2}} + \|\nabla u_n^j u_n^k\|_{L_t^\frac{d+2}{d-2}L_x^\frac{d(d+2)}{d^2-d+2}}
		\longrightarrow 0 \label{112}
\end{align}
as $n \rightarrow \infty$.

\eqref{109} teaches us
\begin{align*}
	& \|u_n^j u_n^k\|_{L_{t,x}^\frac{d+2}{d-2}} \\
		& \leq \|u_n^j - T_n^j\psi_\varepsilon^j\|_{S(\R)}\|u_n^k\|_{S(\R)} + \|T_n^j\psi_\varepsilon^j\|_{S(\R)}\|u_n^k - T_n^k\psi_\varepsilon^k\|_{S(\R)} + \|T_n^j\psi_\varepsilon^j T_n^k\psi_\varepsilon^k\|_{L_{t,x}^\frac{d+2}{d-2}} \\
		& = \|u^j - \psi_\varepsilon^j\|_{S(\R)}\|u^k\|_{S(\R)} + \|\psi_\varepsilon^j\|_{S(\R)}\|u^k - T^k\psi_\varepsilon^k\|_{S(\R)} + \|T_n^j\psi_\varepsilon^j T_n^k\psi_\varepsilon^k\|_{L_{t,x}^\frac{d+2}{d-2}} \\
		& \lesssim \varepsilon + \|T_n^j\psi_\varepsilon^j T_n^k\psi_\varepsilon^k\|_{L_{t,x}^\frac{d+2}{d-2}}.
\end{align*}
So, we estimate the last term.
If $\frac{\lambda_n^j}{\lambda_n^k} + \frac{\lambda_n^k}{\lambda_n^j} \longrightarrow \infty$ as $n \rightarrow \infty$, we have
\begin{align*}
	\|T_n^j\psi_\varepsilon^j T_n^k\psi_\varepsilon^k\|_{L_{t,x}^\frac{d+2}{d-2}}
		& \leq \min\bigl\{\|(T_n^k)^{-1}T_n^j\psi_\varepsilon^j\|_{L_{t,x}^\frac{d+2}{d-2}}\|\psi_\varepsilon^k\|_{L_{t,x}^\infty}, \|\psi_\varepsilon^j\|_{L_{t,x}^\infty}\|(T_n^j)^{-1}T_n^k\psi_\varepsilon^k\|_{L_{t,x}^\frac{d+2}{d-2}}\bigr\} \\
		& \lesssim \min\left\{\(\frac{\lambda_n^k}{\lambda_n^j}\)^\frac{d-2}{2},\(\frac{\lambda_n^j}{\lambda_n^k}\)^\frac{d-2}{2}\right\} \\
		& \longrightarrow 0
\end{align*}
as $n \rightarrow \infty$.
Next, we assume that $\frac{\lambda_n^k}{\lambda_n^j} \longrightarrow \lambda_0 \in (0,\infty)$ and $\frac{|t_n^j(\lambda_n^j)^2 - t_n^k(\lambda_n^k)^2|}{\lambda_n^j\lambda_n^k} \longrightarrow \infty$ as $n \rightarrow \infty$.
Then, we have
\begin{align*}
	\|T_n^j\psi_\varepsilon^j T_n^k\psi_\varepsilon^k\|_{L_{t,x}^\frac{d+2}{d-2}}
		& = \(\frac{\lambda_n^k}{\lambda_n^j}\)^\frac{d-2}{2}\left\|\psi_\varepsilon^j\(\(\frac{\lambda_n^k}{\lambda_n^j}\)^2t + \frac{\lambda_n^k}{\lambda_n^j} \cdot \frac{t_n^j(\lambda_n^j)^2 - t_n^k(\lambda_n^k)^2}{\lambda_n^j\lambda_n^k}, \frac{\lambda_n^k}{\lambda_n^j}x\)\psi_\varepsilon^k\right\|_{L_{t,x}^\frac{d+2}{d-2}} \\
		& \longrightarrow 0 
\end{align*}
as $n \rightarrow \infty$.
Other terms go to zero as $n \rightarrow \infty$ from the similar argument.

We define a function $u_n^{\leq J}$ as
\begin{align*}
	u_n^{\leq J}(t,x)
		:= \sum_{j=1}^J u_n^j(t,x) + e^{it\Delta_\gamma}R_n^J(x).
\end{align*}
We show three conditions to check that $u_n^{\leq J}$ is an approximate solution to \eqref{NLS} in the sense of Lemma \ref{Stability}.

\textbf{(Claim 3).}
\begin{align*}
	\lim_{n \rightarrow \infty}\|u_n^{\leq J}(0) - u_n(0)\|_{\dot{H}_\gamma^1}
		= 0\text{ for any }J.
\end{align*}

It follows from \eqref{107} that
\begin{align*}
	\|u_n^{\leq J}(0) - u_n(0)\|_{\dot{H}_\gamma^1}
		\leq \sum_{j=1}^J \|u_n^j(0) - \phi_n^j\|_{\dot{H}_\gamma^1}
		\longrightarrow 0
\end{align*}
as $n \rightarrow \infty$.

\textbf{(Claim 4).}
\begin{align}
	\sup_{J}\limsup_{n \rightarrow \infty}\|u_n^{\leq J}\|_{\dot{X}^1(\R)}
		\lesssim 1. \label{140}
\end{align}
From \eqref{138}, it suffices to prove
\begin{align*}
	\sup_{J}\limsup_{n \rightarrow \infty}\Bigl\|\sum_{j=1}^J u_n^j\Bigr\|_{\dot{X}^1(\R)}
		\lesssim 1.
\end{align*}
It follows from \eqref{139} and \eqref{112} that
\begin{align}
	\Bigl\|\sum_{j=1}^J u_n^j\Bigr\|_{\dot{X}^1(\R)}^2
		& = \Bigl\|\Bigl(\sum_{j=1}^J \nabla u_n^j\Bigr)^2\Bigr\|_{L_t^\frac{d+2}{d-2}L_x^\frac{d(d+2)}{d^2+4}} \notag \\
		& \lesssim \sum_{j=1}^{J_0}\|u_n^j\|_{\dot{X}^1(\R)}^2 + \sum_{j=J_0+1}^{J}\|u_n^j\|_{\dot{X}^1(\R)}^2 + C(J)\sum_{1\leq j \neq k \leq J}\|\nabla u_n^j \nabla u_n^k\|_{L_t^\frac{d+2}{d-2}L_x^\frac{d(d+2)}{d^2+4}} \notag \\
		& \lesssim \sum_{j=1}^{J_0}\|u_n^j\|_{\dot{X}^1(\R)}^2 + \sum_{j=J_0+1}^J \|\phi^j\|_{\dot{H}_\gamma^1}^2 + o_n(1)
		\lesssim 1. \label{141}
\end{align}

\textbf{(Claim 5).}
\begin{align*}
	\lim_{J \rightarrow \infty}\limsup_{n \rightarrow \infty}\|(i\partial_t + \Delta_\gamma)u_n^{\leq J} + |u_n^{\leq J}|^\frac{4}{d-2}u_n^{\leq J}\|_{\dot{N}^1(\R)}
		= 0.
\end{align*}
We define $F(z) = - |z|^\frac{4}{d-2}z$.
Then, we have
\begin{align*}
	(i\partial_t + \Delta_\gamma)u_n^{\leq J} - F(u_n^{\leq J})
		& = \sum_{j=1}^J F(u_n^j) - F(u_n^{\leq J}) \\
		& = \sum_{j=1}^J F(u_n^j) - F\Bigl(\sum_{j=1}^Ju_n^j\Bigr) + F\Bigl(\sum_{j=1}^Ju_n^j\Bigr) - F(u_n^{\leq J}).
\end{align*}
Thus, it suffices to show that
\begin{align}
	\lim_{J \rightarrow \infty}\limsup_{n \rightarrow \infty}\Bigl\|\sum_{j=1}^J F(u_n^j) - F\Bigl(\sum_{j=1}^Ju_n^j\Bigr)\Bigr\|_{\dot{N}^1(\R)}
		= 0 \label{113}
\end{align}
and
\begin{align}
	\lim_{J \rightarrow \infty}\limsup_{n \rightarrow \infty}\Bigl\|F\Bigl(\sum_{j=1}^Ju_n^j\Bigr) - F(u_n^{\leq J})\Bigr\|_{\dot{N}^1(\R)}
		= 0. \label{114}
\end{align}
First, we see \eqref{113}.
We note that
\begin{align*}
	\Bigl|\nabla \Bigl[\sum_{j=1}^J F(u_n^j) - F\Bigl(\sum_{j=1}^Ju_n^j\Bigr)\Bigr]\Bigr|
		\lesssim_J \sum_{1 \leq j \neq k \leq J}|u_n^k|^\frac{4}{d-2}|\nabla u_n^j|.
\end{align*}
Applying Theorem \ref{Strichartz estimates} and \eqref{112}, we have
\begin{align*}
	\Bigl\|\sum_{j=1}^J F(u_n^j) - F\Bigl(\sum_{j=1}^Ju_n^j\Bigr)\Bigr\|_{\dot{N}^1(\R)}
		& \lesssim \sum_{1 \leq j \neq k \leq J}\||u_n^k|^\frac{4}{d-2}|\nabla u_n^j|\|_{L_t^2L_x^\frac{2d}{d+2}} \\
		& \lesssim \sum_{1 \leq j \neq k \leq J}\|u_n^k\|_{S(\R)}^\frac{6-d}{d-2}\|u_n^k\nabla u_n^j\|_{L_t^\frac{d+2}{d-2}L_x^\frac{d(d+2)}{d^2-d+2}}
		\longrightarrow 0
\end{align*}
as $n \rightarrow \infty$.
Next, we see \eqref{114}.
We note that
\begin{align*}
	& |\nabla(F(u_n^{\leq J} - e^{it\Delta_\gamma}R_n^J) - F(u_n^{\leq J}))| \\
		& \lesssim_J (|u_n^{\leq J}|^\frac{4}{d-2} + |e^{it\Delta_\gamma}R_n^J|^\frac{4}{d-2})|\nabla e^{it\Delta_\gamma}R_n^J| + (|e^{it\Delta_\gamma}R_n^J|^\frac{4}{d-2} + |e^{it\Delta_\gamma}R_n^J||u_n^{\leq J}|^\frac{6-d}{d-2})|\nabla u_n^{\leq J}|.
\end{align*}
and hence, we have
\begin{align*}
	& \|F(u_n^{\leq J} - e^{it\Delta_\gamma}R_n^J) - F(u_n^{\leq J})\|_{\dot{N}^1(\R)} \\
		& \hspace{1.0cm} \lesssim \|u_n^{\leq J}\|_{\dot{X}^1(\R)}(\|u_n^{\leq J}\|_{S(\R)}^\frac{6-d}{d-2}\|e^{it\Delta_\gamma}R_n^J\|_{S(\R)} + \|e^{it\Delta_\gamma}R_n^J\|_{S(\R)}^\frac{4}{d-2}) \\
		& \hspace{2.0cm} + \|e^{it\Delta_\gamma}R_n^J\|_{\dot{X}^1(\R)}\|e^{it\Delta_\gamma}R_n^J\|_{S(\R)}^\frac{4}{d-2} + \|u_n^{\leq J}\|_{S(\R)}^\frac{6-d}{d-2}\|u_n^{\leq J}\nabla e^{it\Delta_\gamma}R_n^J\|_{L_t^\frac{d+2}{d-2}L_x^\frac{d(d+2)}{d^2-d+2}}.
\end{align*}
From Theorem \ref{Strichartz estimates}, \eqref{115}, \eqref{112}, and \eqref{140}, it suffices to prove
\begin{align*}
	& \lim_{J\rightarrow \infty}\limsup_{n \rightarrow \infty}\|u_n^{\leq J}\nabla e^{it\Delta_\gamma}R_n^J\|_{L_t^\frac{d+2}{d-2}L_x^\frac{d(d+2)}{d^2-d+2}} \\
		& = \lim_{J\rightarrow \infty}\limsup_{n \rightarrow \infty}\Bigl\|\Bigl(\sum_{j=1}^J u_n^j + e^{it\Delta_\gamma}R_n^J\Bigr)\nabla e^{it\Delta_\gamma}R_n^J\Bigr\|_{L_t^\frac{d+2}{d-2}L_x^\frac{d(d+2)}{d^2-d+2}} \\
		& \leq \lim_{J\rightarrow \infty}\limsup_{n \rightarrow \infty}\Bigl\{\Bigl\|\Bigl(\sum_{j=1}^J u_n^j\Bigr)\nabla e^{it\Delta_\gamma}R_n^J\Bigr\|_{L_t^\frac{d+2}{d-2}L_x^\frac{d(d+2)}{d^2-d+2}} + \|e^{it\Delta_\gamma}R_n^J\nabla e^{it\Delta_\gamma}R_n^J\|_{L_t^\frac{d+2}{d-2}L_x^\frac{d(d+2)}{d^2-d+2}}\Bigr\}
		= 0.
\end{align*}
The second term is estimated as
\begin{align*}
	\|e^{it\Delta_\gamma}R_n^J\nabla e^{it\Delta_\gamma}R_n^J\|_{L_t^\frac{d+2}{d-2}L_x^\frac{d(d+2)}{d^2-d+2}}
		& \lesssim \|e^{it\Delta_\gamma}R_n^J\|_{\dot{X}^1(\R)}\|\nabla e^{it\Delta_\gamma}R_n^J\|_{S(\R)} \\
		& \lesssim \|R_n^J\|_{\dot{H}^1}\|\nabla e^{it\Delta_\gamma}R_n^J\|_{S(\R)}
		\longrightarrow 0
\end{align*}
as $n \rightarrow \infty$.
Since the argument of \eqref{141} teaches us that given $\eta > 0$, there exists $J_1 = J_1(\eta)$ such that
\begin{align*}
	\sup_{J \geq J_1+1}\limsup_{n \rightarrow \infty}\Bigl\|\sum_{j=J_1+1}^J u_n^j\Bigr\|_{\dot{X}^1(\R)}
		< \eta,
\end{align*}
we show that
\begin{align*}
	\limsup_{n\rightarrow \infty}\|u_n^j\nabla e^{it\Delta_\gamma}R_n^J\|_{L_t^\frac{d+2}{d-2}L_x^\frac{d(d+2)}{d^2-d+2}}
		= 0
\end{align*}
for each $1 \leq j \leq J_1$.
Suppose that $\psi_\varepsilon^j$ has a support on $[-T,T] \times \{|x| \leq R\}$, where the approximate function $\psi_\varepsilon^j \in C_c^\infty(\R \times \R^d)$ is given as \eqref{109}.
Then, we have
\begin{align*}
	& \|u_n^j\nabla e^{it\Delta_\gamma}R_n^J\|_{L_t^\frac{d+2}{d-2}L_x^\frac{d(d+2)}{d^2-d+2}} \\
		& \hspace{1.0cm} \lesssim \|u^j - \psi_\varepsilon^j\|_{S(\R)}\|e^{it\Delta_\gamma}R_n^J\|_{\dot{X}^1(\R)} + \|T_n^j\psi_\varepsilon^j \nabla e^{it\Delta_\gamma}R_n^J\|_{L_t^\frac{d+2}{d-2}L_x^\frac{d(d+2)}{d^2-d+2}} \\
		& \hspace{1.0cm} \lesssim \|u^j - \psi_\varepsilon^j\|_{S(\R)}\|e^{it\Delta_\gamma}R_n^J\|_{\dot{X}^1(\R)} \\
		& \hspace{2.0cm} + (\lambda_n^j)^{-\frac{d}{2}+1}\|\psi_\varepsilon^j\|_{L_{t,x}^\infty}\|\nabla e^{it\Delta_\gamma}R_n^J\|_{L_t^\frac{d+2}{d-2}(B_{(\lambda_n^j)^2T}(-\lambda_n^jt_n^j);L_x^\frac{d(d+2)}{d^2-d+2}(B_{\lambda_n^jR}(0)))} \\
		& \hspace{1.0cm} \lesssim \varepsilon + T^\frac{(d-2)^2}{2(d+2)^2}R^\frac{d^3+4d-16}{2(d+2)^2}\|e^{it\Delta_\gamma}R_n^J\|_{S(\R)}^\frac{d-2}{d+2}\|R_n^J\|_{\dot{H}^1}^\frac{4}{d+2} \\
		& \hspace{2.0cm} + T^\frac{(d-2)^2}{4(d+2)^2}R^\frac{d^3+d^2-12}{2(d+2)^2}\|e^{it\Delta_\gamma}R_n^J\|_{S(\R)}^\frac{d-2}{2(d+2)}\|R_n^J\|_{\dot{H}^1}^\frac{d+6}{2(d+2)}
\end{align*}
from Lemma \ref{Cor of local smoothing}, which implies the desired result.

\textbf{(Step 3).}
We prove convergence.

We have already seen $J^\ast = 1$ in (Step 1) and (Step 2).

We prove $t_n^1 \equiv 0$.
If $t_n^1 \longrightarrow \infty$ as $n \rightarrow \infty$, then
\begin{align*}
	\|e^{it\Delta_\gamma}u_n(0)\|_{S(t \geq 0)}
		\leq \|e^{it\Delta_\gamma}\phi^1\|_{S(t \geq t_n^1)} + \|e^{it\Delta_\gamma}R_n^1\|_{S(t \geq 0)}
		\longrightarrow 0
\end{align*}
as $n \rightarrow \infty$ and hence, we get a contradiction from Lemma \ref{Small data global existence}.
The case of $t_n^1 \longrightarrow - \infty$ as $n \rightarrow \infty$ is excluded by the similar argument.
Thus, we see $t_n^1 \equiv 0$.

We prove $\lim_{n \rightarrow \infty}E_\gamma[\phi_n^1] = E_\gamma[\phi^1] = E_\gamma[u(\tau_n)] (= E_0[W_0])$.
If not, then $(\|u\|_{S(0,\infty)} =) \|u_n\|_{S(\tau_n,\infty)} < \infty$ for sufficiently large $n \in \N$.

Combining this limit and Lemma \ref{Coercivity of energy}, we have
\begin{align*}
	0
		= \lim_{n\rightarrow \infty}\{E_\gamma[\phi_n^1] - E_0[W_0]\}
		= \lim_{n \rightarrow \infty}E_\gamma[R_n^1]
		\gtrsim \lim_{n \rightarrow \infty}\|R_n^1\|_{\dot{H}_\gamma^1}
		\geq 0.
\end{align*}
That is, $\lim_{n\rightarrow\infty}\|R_n^1\|_{\dot{H}_\gamma^1} = 0$ holds and hence, we get
\begin{align*}
	\lim_{n \rightarrow \infty}\|\{u(\tau_n)\}_{[\lambda_n^{-1}]} - \phi^1\|_{\dot{H}\gamma^1}
		 = \lim_{n \rightarrow \infty}\|R_n^1\|_{\dot{H}_\gamma^1}
		 = 0.
\end{align*}
\end{proof}

\begin{lemma}\label{Precompact of critical solution without global existence}
Let $\gamma > 0$ and $u_0 \in BW_+$.
Suppose that the solution $u$ to \eqref{NLS} with \eqref{IC} satisfies $\|u\|_{S([0,T_{\max}))} = \infty$.
Then, there exists $\lambda(t) : [0,T_{\max}) \longrightarrow (0,\infty)$ such that $\{u(t)_{[\lambda(t)]} : t \in [0,T_{\max})$ is precompact in $\dot{H}^1(\R^d)$.
An analogous result holds for negative time direction.
\end{lemma}

\begin{proof}
Applying Lemma \ref{Precompact of critical solution along a sequence}, for each $\{\tau_n\} \subset [0,T_{\max})$ there exists $\{\lambda_n\} \subset (0,\infty)$ satisfying that $\{u(\tau_n)_{[\lambda_n]}\}$ is precompact in $\dot{H}^1(\R^d)$.
Passing to a subsequence, we assume that
\begin{align}
	u(\tau_n)_{[\lambda_n]}
		\longrightarrow u_{\infty,0}\ \text{ in }\ \dot{H}^1(\R^d). \label{117}
\end{align}
We note that $E_\gamma[u_{\infty,0}] = E_0[W_0]$, $\|u_{\infty,0}\|_{\dot{H}_\gamma^1} < \|W_0\|_{\dot{H}^1}$, and $\|u_\infty\|_{S([0,T_{\max}))} = \infty$, where $u_\infty : [0,T_{\max}) \times \R^d \longrightarrow \C$ is the solution to \eqref{NLS} with $u_\infty(0) = u_{\infty,0}$.
For each $t \in [0,T_{\max})$, we define
\begin{align*}
	\lambda(t)
		:= \sup\{\lambda \in (0,\infty) : \|u(t)\|_{\dot{H}_\gamma^1(|x| \leq \frac{1}{\lambda})}^2 = E_0[W_0]\}.
\end{align*}
We note that the following two points :
\begin{itemize}
\item
$\{\lambda \in (0,\infty) : \|u(t)\|_{\dot{H}_\gamma^1(|x| \leq \frac{1}{\lambda})}^2 = E_0[W_0]\}$ is not empty since $\|u(t)\|_{\dot{H}_\gamma^1(|x| \leq \frac{1}{\lambda})}^2 \longrightarrow 0$ as $\lambda \rightarrow \infty$ and $\|u(t)\|_{\dot{H}_\gamma^1(|x| \leq \frac{1}{\lambda})}^2 \longrightarrow \|u(t)\|_{\dot{H}_\gamma^1}^2 \geq 2E_\gamma[u] = 2E_0[W_0]$ as $\lambda \rightarrow 0$.
\item
$\|u(t)\|_{\dot{H}_\gamma^1(|x| \leq \frac{1}{\lambda(t)})}^2 = E_0[W_0]$ holds.
\end{itemize}
From the changing of variables, we have
\begin{align*}
	E_0[W_0]
		= \|u(\tau_n)\|_{\dot{H}_\gamma^1(|x| \leq \frac{1}{\lambda(\tau_n)})}^2
		= \|u(\tau_n)_{[\lambda_n]}\|_{\dot{H}_\gamma^1(|x| \leq \frac{\lambda_n}{\lambda(\tau_n)})}^2.
\end{align*}

Here, we claim that there exists $C > 1$ such that
\begin{align}
	C^{-1}\lambda(\tau_n)
		< \lambda_n
		< C\lambda(\tau_n) \label{116}
\end{align}
for any $n \in \N$.
To prove this claim by contradiction, we consider the cases of $\frac{\lambda_n}{\lambda(\tau_n)} \longrightarrow \infty$ as $n \rightarrow \infty$ and $\frac{\lambda_n}{\lambda(\tau_n)} \longrightarrow 0$ as $n \rightarrow \infty$.

If $\frac{\lambda_n}{\lambda(\tau_n)} \longrightarrow \infty$ as $n \rightarrow \infty$, then
\begin{align*}
	E_0[W_0]
		= \lim_{n\rightarrow \infty}\|u(\tau_n)_{[\lambda_n]}\|_{\dot{H}_\gamma^1(|x| \leq \frac{\lambda_n}{\lambda(\tau_n)})}^2
		= \|u_{\infty,0}\|_{\dot{H}_\gamma^1}^2
		\geq 2E_0[W_0].
\end{align*}

If $\frac{\lambda_n}{\lambda(\tau_n)} \longrightarrow 0$ as $n \rightarrow \infty$, then we get
\begin{align*}
	E_0[W_0]
		= \lim_{n\rightarrow \infty}\|u(\tau_n)_{[\lambda_n]}\|_{\dot{H}_\gamma^1(|x| \leq \frac{\lambda_n}{\lambda(\tau_n)})}^2
		= 0.
\end{align*}
However, these contradict $E_0[W_0] > 0$ and hence, \eqref{116} holds.

\eqref{117} and \eqref{116} imply that $u(\tau_n)_{[\lambda(\tau_n)]}$ converges strongly in $\dot{H}_\gamma^1(\R^d)$.
Since the sequence $\{\tau_n\}$ is taken arbitrarily, $\{u(t)_{[\lambda(t)]} : t \in [0,T_{\max})\}$ is precompact in $\dot{H}_\gamma^1(\R^d)$.
\end{proof}

\begin{lemma}\label{Global existence}
Let $\gamma > 0$ and $u_0 \in BW_+$.
Then, the solution $u$ to \eqref{NLS} with \eqref{IC} exists globally in time.
\end{lemma}

\begin{proof}
We assume $T_{\max} < \infty$ for contradiction.
Then, $\|u\|_{S([0,T_{\max}))} = \infty$ holds from Theorem \ref{Local theory}.
We consider
\begin{align*}
	I_{\mathscr{Y}_R}(t)
		= \int_{\R^d}\mathscr{Y}_R(x) |u(t,x)|^2dx,
\end{align*}
where $\mathscr{Y}_R$ is given as \eqref{127}.
Applying Lemmas \ref{Hardy inequality}, \ref{Virial identity}, and \ref{Coercivity of energy}, we have
\begin{align*}
	|I_{\mathscr{Y}_R}'(t)|
		= \frac{2}{R}\left|\text{Im}\int_{R \leq |x| \leq 2R}|x| \cdot \frac{1}{|x|}\overline{u(t,x)}\nabla u(t,x) \cdot (\nabla \mathscr{Y})\(\frac{x}{R}\)dx\right|
		\lesssim \|u\|_{\dot{H}^1}^2
		\lesssim E_\gamma[u_0]
		= E_0[W_0]
\end{align*}
and hence,
\begin{align}
	|I_{\mathscr{Y}_R}(t) - I_{\mathscr{Y}_R}(T)|
		\lesssim E_0[W_0]|t - T| \label{121}
\end{align}
for any $t$, $T \in (0,T_{\max})$ by the mean value theorem.

We split the rest argument into three parts.

\textbf{(Claim 1).}
\begin{align}
	\lim_{T \rightarrow T_{\max}}\lambda(T)
		= \infty. \label{119}
\end{align}
If not, then there exists a sequence $\{t_n\} \subset I$ with $t_n \longrightarrow T_{\max}$ $(n \rightarrow \infty)$ such that $\lambda(t_n) \longrightarrow \lambda_0 \in [0,\infty)$.

Let $\lambda_0 = 0$.
Lemma \ref{Precompact of critical solution without global existence} implies that $v_n(0) \longrightarrow v_0$ in $\dot{H}^1(\R^d)$ as $n \rightarrow \infty$ for some $v_0 \in \dot{H}^1(\R^d)$ along a subsequence, where the solution $v_n(\tau)$ to \eqref{NLS} is defined as $v_n(\tau) := u(t_n + \frac{\tau}{\lambda(t_n)^2})_{[\lambda(t_n)]}$.
Together with conservation of energy and Lemma \ref{Invariant sets}, we have
\begin{align*}
	E_\gamma[v(t)]
		= E_0[W_0]
	\ \text{ and }\ 
	\|v_0\|_{\dot{H}_\gamma^1}
		\leq \|W_0\|_{\dot{H}^1}.
\end{align*}
From the continuity dependence on initial data, we have
\begin{align*}
	u(0)_{[\lambda(t_n)]}
		= v_n(-\lambda(t_n)^2t_n)
		\longrightarrow v_0
\end{align*}
in $\dot{H}^1(\R^d)$ as $n \rightarrow \infty$.
On the other hand, $u(0)_{[\lambda(t_n)]} \xrightharpoonup[]{\hspace{0.4cm}} 0$ weakly in $\dot{H}^1(\R^d)$ as $n \rightarrow \infty$ from $\lambda_0 = 0$.
Thus, we get $v_0 \equiv 0$, which contradicts $E_\gamma[v_0] = E_0[W_0]$.

Let $\lambda_0 \in (0,\infty)$.
Applying Lemma \ref{Precompact of critical solution without global existence}, $u(t_n)_{[\lambda(t_n)]} \longrightarrow u_\infty$ in $\dot{H}^1(\R^d)$ along a subsequence.
Thus, we have $u(t_n) \longrightarrow f_{[\lambda_0^{-1}]}$ in $\dot{H}^1(\R^d)$.
Theorem \ref{Local theory} teaches us that the solution $u$ exists beyond the maximal existence time $T_{\max}$.

\textbf{(Claim 2).}
\begin{align}
	\lim_{T \rightarrow T_{\max}}I_{\mathscr{Y}_R}(T)
		= 0. \label{120}
\end{align}
Take $r_0 \in (0,1)$.
It follows from Lemma \ref{Hardy inequality} and H\"older inequality that
\begin{align}
	I_{\mathscr{Y}_R}(T)
		& = \int_{|x| \leq r_0}\mathscr{Y}_R(x) |u(T,x)|^2dx + \int_{|x| > r_0}\mathscr{Y}_R(x) |u(T,x)|^2dx \notag \\
		& \lesssim r_0^2\|u(T)\|_{\dot{H}^1}^2 + \|\mathscr{Y}_R\|_{L^\frac{d}{2}}\|u(T)\|_{L^{2^\ast}(|x| \geq r_0)}^2 \notag \\
		& \lesssim r_0^2\|u(T)\|_{\dot{H}^1}^2 + R^2\|\nabla [u(T)_{[\lambda(T)]}]\|_{L^2(|x| \geq \lambda(T)r_0)}^2. \label{118}
\end{align}
Combining precompactness of $\{u(t)_{[\lambda(t)]}\}$ and \eqref{119}, we get that
\begin{align*}
	\text{(The second term of \eqref{118})}
		\longrightarrow 0\ \text{ as }\ T \rightarrow T_{\max}
\end{align*}
for each fixed $r_0 > 0$.
Therefore, we obtain \eqref{120} by taking $T \rightarrow T_{\max}$ and $r_0 \rightarrow 0$.

\textbf{(Claim 3).} Conclusion

Taking $T \rightarrow T_{\max}$ and $R \rightarrow \infty$ in \eqref{121}, we have
\begin{align*}
	\int_{\R^d}|u(t,x)|^2dx
		\lesssim E_0[W_0]|t - T_{\max}|
\end{align*}
from \eqref{120}.
Letting $t \rightarrow T_{\max}$, conservation of mass deduces $u(t,x) = 0$ for any $t \in [0,T_{\max})$.
However, this contradicts $E_\gamma[u_0] = E_0[W_0] > 0$.
\end{proof}

\subsection{Extinction of the soliton-like solution}

In this subsection, we get a contradiction by using the soliton-like solution constructed in the former subsection and complete the proof of scattering part in Theorem \ref{Main theorem}.
In this subsection, we consider mainly the soliton-like solution $u$, that is, $u$ satisfies the following :
\begin{itemize}
\item
$u$ exists globally in positive time,
\item
$u(t)$ belongs to $BW_+$ for each $t \in [0,\infty)$,
\item
$\|u\|_{S([0,\infty))} = \infty$ holds,
\item
There exists $\{\lambda(t)\} \subset (0,\infty)$ such that $\{u(t)_{[\lambda(t)]} : t \in [0,\infty)\}$ is precompact in $\dot{H}_\gamma^1(\R^d)$.
\end{itemize}
As a preliminary to get a contradiction, we state the next lemmas (Lemmas \ref{Concentration on the origin} and \ref{Order of lambda}).
These lemmas hold without $\|u\|_{S([0,\infty))} = \infty$.

\begin{lemma}\label{Concentration on the origin}
Let $\gamma > 0$ and $u_0 \in BW_+$.
Suppose that the solution $u$ to \eqref{NLS} with \eqref{IC} satisfies that $\{u(t)_{[\lambda(t)]} : t \in [0,\infty)\}$ is precompact in $\dot{H}^1(\R^d)$.
Then, for any $\varepsilon > 0$, there exists $R_0 = R_0(\varepsilon) > 0$ such that
\begin{align*}
	\|u(t)_{[\lambda(t)]}\|_{\dot{H}_\gamma^1(|x|\geq R)}^2 + \|u(t)_{[\lambda(t)]}\|_{L^{2^\ast}(|x|\geq R)}^{2^\ast}
		\leq \varepsilon
\end{align*}
for any $t \in \R$ and any $R \geq R_0$.
\end{lemma}

\begin{proof}
This lemma follows from the Arzel\'e-Ascoli theorem.
\end{proof}

\begin{lemma}[Order of $\lambda(t)$]\label{Order of lambda}
Assume that the same conditions with Lemma \ref{Concentration on the origin}.
Then, $\lambda(t)$ satisfies
\begin{align*}
	\lim_{t \rightarrow \infty}\lambda(t)\sqrt{t}
		= \infty.
\end{align*}
\end{lemma}

\begin{proof}
If not, then there exists $\{t_n\} \subset [0,\infty)$ with $t_n \rightarrow \infty$ as $n \rightarrow \infty$ such that $\lambda(t_n)\sqrt{t_n} \longrightarrow \sqrt{\tau_0} < \infty$ as $n \rightarrow \infty$.
We note that $\lambda(t_n) \longrightarrow 0$ as $n \rightarrow \infty$.
Lemma \ref{Precompact of critical solution without global existence} implies that $v_n(0) \longrightarrow v_0$ in $\dot{H}^1(\R^d)$ as $n \rightarrow \infty$ for some $v_0 \in \dot{H}^1(\R^d)$ along a subsequence, where the solution $v_n(\tau)$ to \eqref{NLS} is defined as $v_n(\tau) := u(t_n + \frac{\tau}{\lambda(t_n)^2})_{[\lambda(t_n)]}$.
Together with conservation of energy and Lemma \ref{Invariant sets}, we have
\begin{align*}
	E_\gamma[v(t)]
		= E_0[W_0]
	\ \text{ and }\ 
	\|v_0\|_{\dot{H}_\gamma^1}
		\leq \|W_0\|_{\dot{H}^1}.
\end{align*}
Lemma \ref{Global existence} teaches us that the solution $v$ to \eqref{NLS} with $v(0) = v_0$ exists globally in time.
From the continuity dependence on initial data, we have
\begin{align*}
	u(0)_{[\lambda(t_n)]}
		= v_n(-\lambda(t_n)^2t_n)
		\longrightarrow v(-\tau_0)
\end{align*}
in $\dot{H}^1(\R^d)$ as $n \rightarrow \infty$.
On the other hand, $u(0)_{[\lambda(t_n)]} \xrightharpoonup[]{\hspace{0.4cm}} 0$ weakly in $\dot{H}^1(\R^d)$ as $n \rightarrow \infty$ from $\lambda(t_n) \longrightarrow 0$.
Thus, we get $v(-\tau_0) \equiv 0$, which contradicts $E_\gamma[v_0] = E_0[W_0]$.
\end{proof}

We prove that the soliton-like solution $u$ approaches $W_0$ along a subsequence in the sense of $\delta(t)$.

\begin{lemma}
Let $u$ be the soliton-like solution.
Then, there exists $\{t_n\} \subset [0,\infty)$ with $t_n \longrightarrow \infty$ as $n \rightarrow \infty$ such that $\delta(u(t_n)) \longrightarrow 0$ as $n \rightarrow \infty$.
\end{lemma}

\begin{proof}
We apply Lemma \ref{Precompact of critical solution without global existence} to take $\{\lambda(t)\} \subset (0,\infty)$ satisfying that $\{u(t)_{[\lambda(t)]} : t \in [0,\infty)\}$ is precompact in $\dot{H}^1(\R^d)$.
We define a function
\begin{align*}
	I_{\mathscr{X}_R}(t)
		:= \int_{\R^d}\mathscr{X}_R(x)|u(t,x)|^2dx,
\end{align*}
where $\mathscr{X}_R$ is given as \eqref{126}.
Then, we have
\begin{align*}
	|I_{\mathscr{X}_R}'(t)|
		\lesssim R^2\int_{|x| \leq 3R}\frac{1}{|x|}|u(t,x)||\nabla u(t,x)|dx
		\lesssim R^2\|u(t)\|_{\dot{H}^1}^2
		\lesssim R^2
\end{align*}
for each $R > 0$ by Lemmas \ref{Hardy inequality} and \ref{Coercivity of energy}.
Take any $\varepsilon > 0$.
We set $\sqrt{t}\lambda(t) \geq \frac{R_0(\varepsilon)}{\sqrt{\varepsilon}}$ for each $t \geq t_0 = t_0(\varepsilon)$, where $R_0$ is given in Lemma \ref{Concentration on the origin}.
We note that $t_0$ can be taken by Lemma \ref{Order of lambda}.
If $R = \sqrt{\varepsilon T}$ for $T \geq t_0$, then $\lambda(t)R \geq R_0(\varepsilon)$ holds and hence,
\begin{align}
	I_{\mathscr{X}_R}''(t)
		& = 8\|u(t)\|_{\dot{H}_\gamma^1}^2 - 8\|u(t)\|_{L^{2^\ast}}^{2^\ast} \notag \\
		& \hspace{0.5cm} + 4\int_{|x|\geq R}\left\{\mathscr{X}''\(\frac{|x|}{R}\) - 2\right\}|\nabla u(t,x)|^2dx - \int_{|x|\geq R}(F_{1,R} - 8)|u(t,x)|^{2^\ast}dx \notag \\
		& \hspace{0.5cm} - \int_{R\leq |x|\leq 3R}F_{2,R}|u(t,x)|^2dx + 4\int_{|x| \geq R}\left\{\frac{R}{|x|}\mathscr{X}'\(\frac{|x|}{R}\) - 2\right\}\frac{\gamma}{|x|^2}|u(t,x)|^2dx \label{142} \\
		& \geq 8\|u(t)\|_{\dot{H}_\gamma^1}^2 - 8\|u(t)\|_{L^{2^\ast}}^{2^\ast} - c\left\{\|u(t)\|_{\dot{H}_\gamma^1(|x|\geq R)}^2 + \|u(t)\|_{L^{2^\ast}(|x|\geq R)}^{2^\ast}\right\} \notag \\
		& = 8\|u(t)\|_{\dot{H}_\gamma^1}^2 - 8\|u(t)\|_{L^{2^\ast}}^{2^\ast} - c\left\{\|u(t)_{[\lambda(t)]}\|_{\dot{H}_\gamma^1(|x|\geq \lambda(t)R)}^2 + \|u(t)_{[\lambda(t)]}\|_{L^{2^\ast
}(|x|\geq \lambda(t)R)}^{2^\ast}\right\} \notag \\
		& \geq \frac{16}{d-2}\delta(t) - \varepsilon, \label{146}
\end{align}
where
\begin{align*}
	F_{1,R}(\mathscr{X},|x|)
		& := \frac{4}{d}\left\{\mathscr{X}''\(\frac{|x|}{R}\) + \frac{(d-1)R}{|x|}\mathscr{X}'\(\frac{|x|}{R}\)\right\}, \\
	F_{2,R}(\mathscr{X},|x|)
		& := \frac{1}{R^2}\mathscr{X}^{(4)}\(\frac{|x|}{R}\) + \frac{2(d-1)}{R|x|}\mathscr{X}^{(3)}\(\frac{|x|}{R}\) \\
		& \hspace{1.0cm} + \frac{(d-1)(d-3)}{|x|^2}\mathscr{X}''\(\frac{|x|}{R}\) + \frac{(d-1)(3-d)R}{|x|^3}\mathscr{X}'\(\frac{|x|}{R}\).
\end{align*}
Integrating this inequality over $[t_0,T]$, we have
\begin{align*}
	\int_{t_0}^T \delta(t)dt
		\lesssim I'_{\mathscr{X}_R}(T) - I'_{\mathscr{X}_R}(0) + \varepsilon T
		\lesssim R^2 + \varepsilon T
		\lesssim \varepsilon T.
\end{align*}
Therefore, we obtain
\begin{align*}
	\limsup_{T \rightarrow \infty}\frac{1}{T}\int_0^T \delta(t)dt
		= 0,
\end{align*}
which deduces the desired result.
\end{proof}

We prove that the soliton-like solution $u$ has a distance strictly from $W_0$ in the sense of $\delta$.

\begin{lemma}
Let $u$ be the soliton-like solution.
Then, we have $\inf_{t \in [0,\infty)}\delta(t) > 0$.
\end{lemma}

\begin{proof}
We assume for contradiction that $\inf_{t \in [0,\infty)}\delta(t) = 0$ holds.
Take a sequence $\{t_n\} \subset [0,\infty)$ satisfying $\delta(t_n) \longrightarrow 0$ as $n \rightarrow \infty$.
Since $\{u(t)_{[\lambda(t)]} : t \geq 0\}$ is precompact in $\dot{H}_\gamma^1(\R^d)$, we get
\begin{align*}
	\int_{\R^d}\frac{\gamma}{|x|^2}|u(t_n)_{[\lambda(t_n)]}|^2dx
		\longrightarrow \int_{\R^d}\frac{\gamma}{|x|^2}|\phi(x)|^2dx
\end{align*}
as $n \rightarrow \infty$ along some subsequence of $\{t_n\}$.
On the other hand, it follows from Lemma \ref{Parameters for modulation} that
\begin{align*}
	\[\int_{\R^d}\frac{\gamma}{|x|^2}|[u(t_n,x)]_{[\lambda(t_n)]}|^2dx\]^\frac{1}{2}
		= \[\int_{\R^d}\frac{\gamma}{|x|^2}|u(t_n,x)|^2dx\]^\frac{1}{2}
		\lesssim |\delta(t_n)|
		\longrightarrow 0
\end{align*}
as $n \rightarrow \infty$.
Thus, we obtain $\phi \equiv 0$.
However, this contradicts $E_\gamma[\phi] = E_\gamma[u] = E_0[W_0]$.
\end{proof}

\section{Blow-up}\label{Sec: Blow-up}

In this subsection, we prove blow-up part in Theorem \ref{Main theorem}.

\begin{lemma}\label{Exponential decay}
Let $\gamma > 0$ and $u_0 \in L^2(\R^d) \cap BW_-$.
If the solution $u$ to \eqref{NLS} with \eqref{IC} exists globally in positive time, then there exists $C_1$, $C_2 > 0$ such that
\begin{align}
	\int_t^\infty |\delta(\tau)|d\tau
		\leq C_1 e^{-C_2t} \label{143}
\end{align}
for any $t \in (0,\infty)$.
An analogous result holds for negative time.
\end{lemma}

\begin{proof}
We split the proof into two steps.

\textbf{(Step 1).}
There exist $C$, $R_0 > 0$ such that
\begin{align}
	I_{\mathscr{X}_R}'(t)
		\leq CR^2|\delta(t)|
	\ \text{ and }\ 
	I_{\mathscr{X}_R}''(t)
		\leq - \frac{8}{d-2}|\delta(t)| \label{147}
\end{align}
for any $t \geq 0$ and any $R \geq R_0$, where $I_{\mathscr{X}_R}$ is defined in Lemma \ref{Virial identity}.

Using expansion \eqref{145} of $u$, we have
\begin{align*}
	I_{\mathscr{X}_R}'(t)
		& = 2R^2\text{Im}\int_{\R^d}\frac{\overline{w + W_0}}{R\mu(t)}\nabla(w + W_0) \cdot (\nabla \mathscr{X})\(\frac{y}{R\mu(t)}\)dy \\
		& = 2R^2\text{Im}\int_{\R^d}\frac{W_0\nabla w + \overline{w}\nabla W_0 + \overline{w}\nabla w}{R\mu(t)} \cdot (\nabla \mathscr{X})\(\frac{y}{R\mu(t)}\)dy.
\end{align*}
Lemma \ref{Parameters for modulation} deduces that
\begin{align*}
	I_{\mathscr{X}_R}'(t)
		& = 2R^2\text{Im}\int_{|y| \leq 3R\mu(t)}\frac{|y|}{R\mu(t)} \left\{\frac{W_0}{|y|}\nabla w + \frac{\overline{w}}{|y|}\nabla W_0 + \frac{\overline{w}}{|y|}\nabla w\right\} \cdot (\nabla \mathscr{X})\(\frac{y}{R\mu(t)}\)dy \\
		& \lesssim R^2\{\|W_0\|_{\dot{H}^1}\|w\|_{\dot{H}^1} + \|w\|_{\dot{H}^1}^2\}\|\nabla \mathscr{X}\|_{L^\infty} \\
		& \lesssim R^2(\|w\|_{\dot{H}^1}^2 + \|w\|_{\dot{H}^1}) \\
		& \lesssim R^2\{|\delta(t)|^2 + |\delta(t)|\}
		\lesssim R^2|\delta(t)|.
\end{align*}
Using \eqref{142}, we have
\begin{align*}
	I_{\mathscr{X}_R}''(t)
		:= \frac{16}{d-2}\delta(t) + A_R(u(t)).
\end{align*}
We estimate $A_R(u(t))$.
It follows from Lemma \ref{Radial Sobolev inequality} and mass conservation that
\begin{align*}
	A_R(u(t))
		& \leq \frac{c}{R^2}\|u(t)\|_{L^2}^2 + c\,\|u(t)\|_{L^{2^\ast}(|x| \geq R)}^{2^\ast} \\
		& \leq \frac{c}{R^2}\|u(t)\|_{L^2}^2 + \frac{c}{R^\frac{2(d-1)}{d-2}}\|u(t)\|_{L^2}^\frac{2(d-1)}{d-2}\|u(t)\|_{\dot{H}^1}^\frac{2}{d-2} \\
		& \leq \frac{c}{R^2} + \frac{c}{R^\frac{2(d-1)}{d-2}}(-\delta(t) + \|W_0\|_{\dot{H}^1}^2)^\frac{1}{d-2}.
\end{align*}

When $|\delta(t)| \geq \delta_1 > 0$ (which is chosen later), there exists $R_1(\delta_1) > 0$ (independent of $t$) such that
\begin{align*}
	A_R(u(t))
		& \lesssim \(\frac{1}{R^2} + \frac{1}{R^\frac{2(d-1)}{d-2}}\)\delta_1 + \frac{1}{R^\frac{2(d-1)}{d-2}}|\delta(t)|^\frac{1}{d-2} \\
		& \lesssim \(\frac{1}{R^2} + \frac{1}{R^\frac{2(d-1)}{d-2}}\)|\delta(t)| + \frac{1}{R^\frac{2(d-1)}{d-2}}\delta_1^{-\frac{d-3}{d-2}}|\delta(t)|
		\leq \frac{8}{d-2}|\delta(t)|
\end{align*}
for any $R > R_1$.

When $|\delta(t)| < \delta_1 (< \delta_0)$, mass conservation and Lemma \ref{Parameters for modulation} gives us that
\begin{align*}
	M[u_0]
		& = M[u(t)]
		\geq \|u(t)\|_{L^2(|x| \leq \mu(t)^{-1})}^2
		= \mu(t)^{-2}\|u(t)_{[\theta(t),\mu(t)]}\|_{L^2(|x|\leq 1)}^2 \\
		& = \mu(t)^{-2}\|W_0 + w(t)\|_{L^2(|x|\leq 1)}^2
		\geq \mu(t)^{-2}\{\|W_0\|_{L^2(|x| \leq 1)} - \|w(t)\|_{L^2(|x|\leq 1)}\}^2 \\
		& \geq \mu(t)^{-2}\{\|W_0\|_{L^2(|x|\leq 1)} - c\,\|w(t)\|_{\dot{H}^1}\}^2
		\geq \mu(t)^{-2}\{\|W_0\|_{L^2(|x|\leq 1)} - c\delta_1\}^2
\end{align*}
and hence, we have
\begin{align*}
	\inf\{\mu(t) : t \geq 0, |\delta(t)| \leq \delta_1\}
		> 0.
\end{align*}
Since
\begin{align*}
	A_R(W_0)
		= 4\int_{|x| \geq R}\left\{\frac{R}{|x|}\mathscr{X}'\(\frac{|x|}{R}\) - 2\right\}\frac{\gamma}{|x|^2}|W_0(x)|^2dx
		\leq 0,
\end{align*}
we have
\begin{align}
	A_R(u(t))
		& = A_R((W_0 + w)_{[\mu(t)^{-1}]})
		= A_{R\mu(t)}(W_0 + w) \notag \\
		& \leq A_{R\mu(t)}(W_0 + w) - A_{R\mu(t)}(W_0) \notag \\
		\begin{split}\label{148}
		& \lesssim \int_{|x|\geq R\mu(t)}\[|\nabla w|^2 + |\nabla W_0 \cdot \nabla w| + W_0^{2^\ast-1}|w| + |w|^{2^\ast}\]dx \\
		& \hspace{1.0cm} + \frac{1}{(R\mu(t))^2}\int_{R\mu(t) \leq |x| \leq 2R\mu(t)}\[W_0|w| + |w|^2\]dx
		\end{split} \\
		& \lesssim \|w\|_{\dot{H}^1}^2 + \|w\|_{\dot{H}^1}^{2^\ast} + (R\mu(t))^\frac{2-d}{2}\|w\|_{\dot{H}^1} + (R\mu(t))^{-\frac{d+2}{2}}\|w\|_{\dot{H}^1} \notag \\
		& \lesssim (|\delta(t)| + |\delta(t)|^{2^\ast-1} + R^\frac{2-d}{2} + R^{-\frac{d+2}{2}})|\delta(t)|, \notag
\end{align}
where the third term in \eqref{148} is estimated as
\begin{align*}
	\int_{|x|\geq R\mu(t)}W_0^{2^\ast-1}|w|dx
		\lesssim \|xW_0^{2^\ast-1}\|_{L^2(|x|\geq R\mu(t))}\|w\|_{\dot{H}^1}
		\lesssim (R\mu(t))^{-\frac{d+2}{2}}\|w\|_{\dot{H}^1}
\end{align*}
and the second and fifth terms can be estimated similarly.
Then, if we take $\delta_1 > 0$ sufficiently small and $R_2 > 0$ sufficiently large, we obtain
\begin{align*}
	A_R(u(t))
		\leq \frac{8}{d-2}|\delta(t)|
\end{align*}
for any $R \geq R_2$ and $t \geq 0$ with $|\delta(t)| \leq \delta_1$.
Therefore, $I_{\mathscr{X}_R}''(t) \leq - \frac{8}{d-2}|\delta(t)|$ for any $t \geq 0$ and $R \geq R_0 := \max\{R_1,R_2\}$.

\textbf{(Step 2).}
Conclusion

We prove that $I_{\mathscr{X}_R}'(t) > 0$ for given $R \geq R_0$ and $t \geq 0$.
Contrary, we assume that there exists $t_0 \geq 0$ such that $I_{\mathscr{X}_R}'(t_0) \leq 0$.
Since we saw that $I_{\mathscr{X}_R}'$ is strictly decreasing in Step 1, we have $I_{\mathscr{X}_R}'(t) \leq I_{\mathscr{X}_R}'(t_0 + 1) < 0$ for $t \geq t_0 + 1$.
However, this implies that $I_{\mathscr{X}_R}(t) < 0$ for sufficiently large $t$.
Using the fundamental theorem of Calculus and the results in Step 1 \eqref{147}, 
\begin{align*}
	I_{\mathscr{X}_R}'(T) - I_{\mathscr{X}_R}'(t)
		= \int_t^T I_{\mathscr{X}_R}''(s)ds
		\leq - \frac{8}{d-2}\int_t^T |\delta(s)|ds
\end{align*}
and
\begin{align*}
	I_{\mathscr{X}_R}'(T) - I_{\mathscr{X}_R}'(t)
		> - I_{\mathscr{X}_R}'(t)
		\geq - CR^2|\delta(t)|
\end{align*}
for any $T > t > 0$.
Combining these formulas
\begin{align*}
	\int_t^T |\delta(\tau)| d\tau
		\leq CR^2|\delta(t)|
\end{align*}
for fixed $R > R_0$.
Letting $T \rightarrow \infty$, we have
\begin{align*}
	\int_t^\infty |\delta(\tau)| d\tau
		\leq C|\delta(t)|.
\end{align*}
Gr\"onwall's inequality gives us the desired result.
\end{proof}

We prove that time global solutions $u$ belonging to $BW_-$ approaches $W_0$ in the sense of $\delta(t)$.

\begin{lemma}
Suppose that the same assumptions as Lemma \ref{Exponential decay} hold.
Then, we have $\delta(t) \longrightarrow 0$ as $t \rightarrow \infty$.
\end{lemma}

\begin{proof}
From \eqref{143}, there exists a sequence $\{t_n\}$ such that $t_n \longrightarrow \infty$ and $\delta(t_n) \longrightarrow 0$ as $n \rightarrow \infty$ as $n \rightarrow \infty$.
To use contradiction, we take a sequence $\{s_n\}$ satisfying that
\begin{align*}
	t_n < s_n, \quad
	- \delta(s_n) = \varepsilon_0, \quad
	- \delta(t) < \varepsilon_0\ \text{ for any }t \in [t_n,s_n)
\end{align*}
for some $\varepsilon_0 \in (0,\delta_0)$.
We check that $\lim_{t \rightarrow \infty}\delta(t) = 0$ under the assumption
\begin{align}
	\sup_{n \in \N}\sup_{t \in [t_n,s_n]}\mu(t)
		=: \sup_{n \in \N}\mu(b_n)
		< \infty. \label{144}
\end{align}
It follows form the fundamental of Calculus, Lemma \ref{Parameters for modulation2}, \eqref{144}, and Lemma \ref{Exponential decay} that
\begin{align*}
	|\alpha(s_n) - \alpha(t_n)|
		\leq \int_{t_n}^{s_n}|\alpha'(t)|dt
		\lesssim \int_{t_n}^{s_n} \mu(t)^2|\delta(t)|dt
		\lesssim \int_{t_n}^\infty |\delta(t)|dt
		\longrightarrow 0.
\end{align*}
On the other hand, we have
\begin{align*}
	\lim_{n \rightarrow \infty}|\alpha(s_n) - \alpha(t_n)|
		= \lim_{n \rightarrow \infty}|\alpha(s_n)|
		= O(\varepsilon_0)
\end{align*}
from the construction of $\{t_n\}$ and $\{s_n\}$.
Thus, it suffices to prove \eqref{144} and we devote the rest of proof to it. 

We first show that $\sup_{n \in \N}\mu(t_n) < \infty$.
We assume for contradiction that $\mu(t_n) \longrightarrow \infty$ as $n \rightarrow \infty$.
It follows from Lemma \ref{Parameters for modulation} that
\begin{align*}
	\|w(t_n)\|_{\dot{H}^1}
		\sim |\alpha(t_n)|
		\sim |\delta(t_n)|
		\longrightarrow 0
\end{align*}
and hence, we have $u(t_n)_{[\theta(t_n),\mu(t_n)]} \longrightarrow W_0$ in $\dot{H}^1(\R^d)$.
This limit implies that
\begin{align*}
	\|u(t_n)\|_{L^{2^\ast}(|x| \geq \varepsilon)}^{2^\ast}
		= \|u(t_n)_{[\theta(t_n),\mu(t_n)]}\|_{L^{2^\ast}(|x| \geq \mu(t_n)\varepsilon)}^{2^\ast}
			\longrightarrow 0
\end{align*}
as $n \rightarrow \infty$ for each $\varepsilon > 0$.
Combining this limit and H\"older inequality, we get
\begin{align*}
	I_{\mathscr{X}_R}(t_n)
		& = \int_{|x| \geq \varepsilon}\mathscr{X}_R(x)|u(t_n,x)|^2dx + \int_{|x| \leq \varepsilon}\mathscr{X}_R(x)|u(t_n,x)|^2dx \\
		& \lesssim R^4\|u(t_n)\|_{L^{2^\ast}(|x|\geq \varepsilon)}^2 + o_\varepsilon(1)
		\longrightarrow 0
\end{align*}
as $n \rightarrow \infty$ and $\varepsilon \rightarrow 0$.
Recalling $I_{\mathscr{X}_R}'(t) > 0$, $I_{\mathscr{X}_R}(t) < 0$ holds for $t \geq 0$, which is a contradiction.
Applying the fundamental theorem of Calculus, Lemmas \ref{Parameters for modulation2} and \ref{Exponential decay}, we have
\begin{align*}
	\left|\frac{1}{\mu(t_n)^2} - \frac{1}{\mu(b_n)^2}\right|
		\lesssim \int_{t_n}^{s_n}\left|\frac{\mu'(t)}{\mu(t)^3}\right|dt
		\lesssim \int_{t_n}^\infty|\delta(t)|dt
		\longrightarrow 0
\end{align*}
as $n \rightarrow \infty$.
Together with $\sup_{n \in \N}\mu(t_n) < \infty$, we obtain \eqref{144}.
\end{proof}

We prove that time global solutions $u$ belonging to $BW_-$ has a distance from $W_0$ in the sense of $\delta(t)$.

\begin{lemma}
Then, we have $\inf_{t \in [0,\infty)}|\delta(t)| > 0$.
\end{lemma}

\begin{proof}
It suffices to prove that $\inf_{t \in I_0}|\delta(t)| > 0$.
Then, we recall the expansion in Proposition \ref{Modulation}
\begin{align*}
	u(t)_{[\theta(t),\mu(t)]}
		= (1 + \alpha(t))W_0 + v(t),
	\quad t \in I_0.
\end{align*}
Then, Lemma \ref{Parameters for modulation} deduces that
\begin{align*}
	\int_{|x| \leq 1}|W_0(x)|^2dx
		& \lesssim \int_{|x| \leq 1}(|u_{[\theta(t),\mu(t)]}(t,x)|^2 + |\alpha(t)|^2|W_0(x)|^2 + |v(t,x)|^2)dx \\
		& \lesssim \int_{\R^d}\frac{\gamma}{|x|^2}(|u(t,x)|^2 + |v(t,x)|^2)dx + |\alpha(t)|^2 \\
		& \lesssim \int_{\R^d}\frac{\gamma}{|x|^2}|u(t,x)|^2dx + \|v(t)\|_{\dot{H}^1}^2 + |\alpha(t)|^2
		\lesssim |\delta(t)|^2.
\end{align*}
\end{proof}

\subsection*{Acknowledgements}
M.H. is supported by JSPS KAKENHI Grant Number JP22J00787.
M.I. is supported by JSPS KAKENHI Grant Number JP19K14581 and JST CREST Grant Number JPMJCR1913.


\begin{thebibliography}{99}
\bibitem{AkaNaw13}
T. Akahori and H. Nawa, \textit{Blowup and scattering problems for the nonlinear Schr\"odinger equations}, Kyoto J. Math. \textbf{53} (2013), no. 3, 629--672. MR3102564

\bibitem{ArdHamIke22}
Alex H. Ardila, M. Hamano, and M. Ikeda, \textit{Mass-energy threshold dynamics for the focusing NLS with a repulsive inverse-power potential}, preprint, arXiv : 2202.11640.

\bibitem{ArdInu22}
Alex H. Ardila and T. Inui, \textit{Threshold scattering for the focusing NLS with a repulsive Dirac delta potential}, J. Differential Equations \textbf{313} (2022), 54--84. MR4362369

\bibitem{AroDodMur20}
A. K. Arora, B. Dodson, and J. Murphy, \textit{Scattering below the ground state for the 2d radial nonlinear Schr\"odinger equation}, Proc. Amer. Math. Soc. \textbf{148} (2020), no. 4, 1653--1663. MR4069202


\bibitem{BurPlaStaTah03}
N. Burq, F. Planchon, J. Stalker, and A. S. Tahvildar-Zadeh, \textit{Strichartz estimates for the wave and Schr\"odinger equations with the inverse-square potential}, J. Funct. Anal. \textbf{203} (2003), no. 2, 519--549. MR2003358

\bibitem{CamFarRou22}
L. Campos, L. G. Farah, and S. Roudenko, \textit{Threshold solutions for the nonlinear Schr\"odinger equation} Rev. Mat. Iberoam. \textbf{38} (2022), no. 5, 1637--1708. MR4502077

\bibitem{CamMur22}
L. Campos and J. Murphy, \textit{Threshold solutions for the intercritical inhomogeneous NLS}, preprint, arXiv : 2205.09714.

\bibitem{CamPas22}
L. Campos and A. Pastor, \textit{Threshold solutions for cubic Schr\"odinger systems}, preprint, arXiv : 2210.07369.

\bibitem{Caz03}
T. Cazenave, \textit{Semilinear Schr\"odinger equations}, Courant Lecture Notes in Mathematics, 10. New York University, Courant Institute of Mathematical Sciences, New York; American Mathematical Society, Providence, RI, 2003. xiv+323 pp. MR2002047


\bibitem{Dod15}
B. Dodson, \textit{Global well-posedness and scattering for the mass critical nonlinear Schr\"odinger equation with mass below the mass of the ground state}, Adv. Math. 285 (2015), 1589--1618. MR3406535

\bibitem{Dod19}
B. Dodson, \textit{Global well-posedness and scattering for the focusing, cubic Schr\"odinger equation in dimension $d = 4$}, Ann. Sci. \'Ec. Norm. Sup\'er. (4) \textbf{52} (2019), no. 1, 139--180. MR3940908

\bibitem{DodMur17}
B. Dodson and J. Murphy, \textit{A new proof of scattering below the ground state for the 3D radial focusing cubic NLS}, Proc. Amer. Math. Soc. \textbf{145} (2017), no. 11, 4859--4867. MR3692001

\bibitem{DodMur18}
B. Dodson and J. Murphy, \textit{A new proof of scattering below the ground state for the non-radial focusing NLS}, Math. Res. Lett. \textbf{25} (2018), no. 6, 1805--1825. MR3934845

\bibitem{DuWuZha16}
D. Du, Y. Wu, and K. Zhang, \textit{On blow-up criterion for the nonlinear Schr\"odinger equation}, Discrete Contin. Dyn. Syst. \textbf{36} (2016), no. 7, 3639--3650. MR3485846

\bibitem{DuyLanRou22}
T. Duyckaerts, O. Landoulsi, and S. Roudenko, \textit{Threshold solutions in the focusing 3D cubic NLS equation outside a strictly convex obstacle}, J. Funct. Anal. \textbf{282} (2022), no. 5, Paper No. 109326, 55 pp. MR4352607

\bibitem{DuyHolRou10}
T. Duyckaerts, J. Holmer, and S. Roudenko, \textit{Scattering for the non-radial 3D cubic nonlinear Schr\"odinger equation}, Math. Res. Lett. \textbf{15} (2008), no. 6, 1233--1250. MR2470397

\bibitem{DuyMer09}
T. Duyckaerts and F. Merle, \textit{Dynamic of threshold solutions for energy-critical NLS}, Geom. Funct. Anal. \textbf{18} (2009), no. 6, 1787--1840. MR2491692

\bibitem{DuyRou10}
T. Duyckaerts and S. Roudenko, \textit{Threshold solutions for the focusing 3D cubic Schr\"odinger equation}, Rev. Mat. Iberoam. \textbf{26} (2010), no. 1, 1--56. MR2662148

\bibitem{FanXieCaz11}
D. Fang, J. Xie, and T. Cazenave, \textit{Scattering for the focusing energy-subcritical nonlinear Schr\"odinger equation}, Sci. China Math. \textbf{54} (2011), no. 10, 2037--2062. MR2838120

\bibitem{Gla77}
R. T. Glassey, \textit{On the blowing up of solutions to the Cauchy problem for nonlinear Schr\"odinger equations}, J. Math. Phys. \textbf{18} (1977), no. 9, 1794--1797. MR0460850

\bibitem{Guo11}
Q. Guo, \textit{Threshold solutions for the focusing $L^2$-supercritical NLS Equations}, preprint, arXiv : 1111.5669.

\bibitem{GusInu221}
S. Gustafson and T. Inui, \textit{Threshold odd solutions to the nonlinear Schr\"odinger equation in one dimension}, Partial Differ. Equ. Appl. 3 (2022), no. 4, Paper No. 46, 45 pp. MR4447416

\bibitem{GusInu22}
S. Gustafson and T. Inui, \textit{Blow-up or Grow-up for the threshold solutions to the nonlinear Schr\"odinger equation}, preprint, arXiv : 2209.04767.

\bibitem{GusInu222}
S. Gustafson and T. Inui, \textit{Threshold even solutions to the nonlinear Schr\"odinger equation with delta potential at high frequencies}, preprint, arXiv : 2211.15591.

\bibitem{HamKikWat22}
M. Hamano, H. Kikuchi, and M. Watanabe, \textit{Threshold solutions for the 3D focusing cubic-quintic nonlinear Schrodinger equation at low frequencies}, preprint, arXiv : 2210.08201.

\bibitem{HolRou08}
J. Holmer and S. Roudenko, \textit{A sharp condition for scattering of the radial 3D cubic nonlinear Schr\"odinger equation}, Comm. Math. Phys. \textbf{282} (2008), no. 2, 435--467. MR2421484

\bibitem{HolRou12}
J. Holmer and S. Roudenko, \textit{Divergence of infinite-variance nonradial solutions to the 3D NLS equation}, Comm. Partial Differential Equations \textbf{35} (2010), no. 5, 878--905. MR2753623

\bibitem{Kai17}
Y. Kai, \textit{Dynamics of the Energy Critical Nonlinear Schrodinger Equation with Inverse Square Potential}, Thesis (Ph.D.)-The University of Iowa. 2017. 118 pp. MR3705915

\bibitem{Kai20}
Y. Kai, \textit{Scattering of the energy-critical NLS with inverse square potential}, J. Math. Anal. Appl. \textbf{487} (2020), no. 2, 124006, 22 pp. MR4074198

\bibitem{Kai21}
Y. Kai, \textit{Scattering of the focusing energy-critical NLS with inverse square potential in the radial case}, Commun. Pure Appl. Anal. \textbf{20} (2021), no. 1, 77--99. MR4191497

\bibitem{KaiZenZha22}
Y. Kai, C. Zeng, and X. Zhang, \textit{Dynamics of threshold solutions for energy critical NLS with inverse square potential}, SIAM J. Math. Anal. \textbf{54} (2022), no. 1, 173--219. MR4358028

\bibitem{KalSchWalWus75}
H. Kalf, U. W. Schmincke, J. Walter, and R. W\"ust, \textit{On the spectral theory of Schr\"odinger and Dirac operators with strongly singular potentials}, Spectral theory and differential equations (Proc. Sympos., Dundee, 1974; dedicated to Konrad J\"orgens), pp. 182--226. Lecture Notes in Math., Vol. \textbf{448}, Springer, Berlin, 1975. MR0397192

\bibitem{KenMer06}
C. E. Kenig and F. Merle, \textit{Global well-posedness, scattering and blow-up for the energy-critical, focusing, non-linear Schr\"odinger equation in the radial case}, Invent. Math. \textbf{166} (2006), no. 3, 645--675. MR2257393

\bibitem{KilMiaVisZhaZhe17}
R. Killip, C. Miao, M. Visan, J. Zhang, and J. Zheng, \textit{The energy-critical NLS with inverse-square potential}, Discrete Contin. Dyn. Syst. \textbf{37} (2017), no. 7, 3831--3866. MR3639442

\bibitem{KilMiaVisZhaZhe18}
R. Killip, C. Miao, M. Visan, J. Zhang, and J. Zheng, \textit{Sobolev spaces adapted to the Schr\"odinger operator with inverse-square potential}, Math. Z. \textbf{288} (2018), no. 3-4, 1273--1298. MR3778997

\bibitem{KilMurVisZhe17}
R. Killip, J. Murphy, M. Visan, and J. Zheng, \textit{The focusing cubic NLS with inverse-square potential in three space dimensions}, Differential Integral Equations \textbf{30} (2017), no. 3-4, 161--206. MR3611498

\bibitem{KilVis10}
R. Killip and M. Visan, \textit{The focusing energy-critical nonlinear Schr\"odinger equation in dimensions five and higher}, Amer. J. Math. \textbf{132} (2010), no. 2, 361--424. MR2654778

\bibitem{KilVis13}
R. Killip and M. Visan, \textit{Nonlinear Schr\"odinger equations at critical regularity}, Evolution equations, 325--437, Clay Math. Proc., \textbf{17}, Amer. Math. Soc., Providence, RI, 2013. MR3098643

\bibitem{LiZha09}
D. Li and X. Zhang, \textit{Dynamics for the energy critical nonlinear Schr\"odinger equation in high dimensions}, J. Funct. Anal. \textbf{256} (2009), no. 6, 1928--1961. MR2498565


\bibitem{LuMiaMur18}
J. Lu, C. Miao, and J. Murphy, \textit{Scattering in $H^1$ for the intercritical NLS with an inverse-square potential}, J. Differential Equations \textbf{264} (2018), no. 5, 3174--3211. MR3741387

\bibitem{MiaMurZhe21}
C. Miao, J. Murphy, and J. Zheng, \textit{Threshold scattering for the focusing NLS with a repulsive potential}, preprint. arXiv : 2102.07163.

\bibitem{MiaSuZhe22}
C. Miao, X. Su, and J. Zheng, \textit{The $W^{s,p}$-boundedness of stationary wave operators for the Schr\"odinger operator with inverse-square potential}, preprint, arXiv : 2110.01969.

\bibitem{OgaTsu911}
T. Ogawa and Y. Tsutsumi, \textit{Blow-up of $H^1$ solutions for the one-dimensional nonlinear Schr\"odinger equation with critical power nonlinearity}, Proc. Amer. Math. Soc. \textbf{111} (1991), no. 2, 487--496. MR1045145

\bibitem{OgaTsu912}
T. Ogawa and Y. Tsutsumi, \textit{Blow-up of H1 solution for the nonlinear Schr\"odinger equation}, J. Differential Equations \textbf{92} (1991), no. 2, 317--330. MR1120908

\bibitem{Rey90}
O. Rey, \textit{The role of the Green's function in a nonlinear elliptic equation involving the critical Sobolev exponent}, J. Funct. Anal. \textbf{89} (1990), no. 1, 1--52. MR1040954

\bibitem{Str77}
W. A. Strauss, \textit{Existence of solitary waves in higher dimensions}, Comm. Math. Phys. \textbf{55} (1977), no. 2, 149--162. MR0454365


\bibitem{Zhe18}
J. Zheng, \textit{Focusing NLS with inverse square potential}, J. Math. Phys. \textbf{59} (2018), no. 11, 111502, 14 pp. MR3872306
\end{thebibliography}
\end{document}